\documentclass[a4paper,11pt,twoside,reqno]{article}

\usepackage{comment}
\usepackage[english]{babel}
\usepackage{fancyhdr}
\usepackage{fullpage}
\usepackage[dvips]{graphicx}
\usepackage[T1]{fontenc}
\usepackage{fouriernc}
\usepackage[latin1,applemac]{inputenc}
\usepackage{amsmath}
\usepackage{amsfonts}
\usepackage{amssymb}
\usepackage{mathrsfs}
\usepackage{amsthm}
\usepackage{latexsym}
\usepackage{array}
\usepackage{enumerate}
\usepackage{mathrsfs}
\usepackage{amsxtra}
\usepackage{amscd}
\usepackage{mathtools}
\usepackage[usenames]{color}
\definecolor{red}{rgb}{1.0,0.0,0.0}

\definecolor{blu}{rgb}{0.0,0.0,1.0}

\definecolor{violet}{rgb}{0.5,0.0,0.4}

\usepackage[usenames]{color}
\def\eps{\varepsilon}

\newtheorem{Theorem}{Theorem}[section]
\newtheorem{Definition}[Theorem]{Definition}
\newtheorem{Proposition}[Theorem]{Proposition}
\newtheorem{Lemma}[Theorem]{Lemma}
\newtheorem{Corollary}[Theorem]{Corollary}

\newtheorem{Assumption}[Theorem]{Assumption}
\newtheorem{Remark}[Theorem]{Remark}
\newtheorem{Example}[Theorem]{Example}

\numberwithin{equation}{section}

\newcommand{\myref}[1]{(\ref {#1})}

\def\R{\mathbb R}
\def\N{\mathbb N}
\def\E{\mathbb E}

\def\cald{{\mathcal D}}
\def\cala{{\mathcal A}}
\def\calb{{\mathcal B}}
\def\cale{{\mathcal E}}

\def\calk{{\mathcal K}}
\def\call{{\mathcal L}}

\def\calp{{\mathcal P}}

\def\<{\left\langle }
\def\>{\right\rangle }

\def\Swiech
{{\accent1 S}wie{\hbox{\kern -0.21em\lower
0.77ex\hbox{$\textfont1=\scriptfont1
\lhook$}}}ch}

\def\SWIECH
{{\accent1 S}WIE{\hbox{\kern -0.21em\lower
0.77ex\hbox{$\textfont1=\scriptfont1
\lhook$}}}CH}

\newtheoremstyle{mytheorem}% name
{6pt}%Space above
{6pt}%Space below
{\itshape}%Body font
{-0pt}%Indent amount 1
{\large \scshape}% Theorem head font
{}%Punctuation after theorem head
{1em}%Space after theorem head 2
{}%Theorem head spec (can be left empty, meaning "normal")

\newtheoremstyle{myremark}% name
{6pt}%Space above
{10pt}%Space below
{\rm}%Body font
{-0pt}%Indent amount 1
{\large \scshape}% Theorem head font
{}%Punctuation after theorem head
{1em}%Space after theorem head 2
{}%Theorem head spec (can be left empty, meaning "normal")

\makeatletter
\def\@endtheorem{{\hfill\hbox{\enspace${ \blacksquare}$}}\endtrivlist\@endpefalse } % insert `\qed` macro
\makeatother

\def\eps{\varepsilon}

\def\calf{{\mathcal{F}}}

\def\calp{{\mathcal{P}}}
\def\calr{{\mathcal{R}}}

\def\dis{\displaystyle}

\def\eps{\varepsilon}

\def\P{\mathfrak{P}}

\def\N{\mathcal{N}}

\def\norm{{\| \kern -.05em | }}

\newcommand{\reals}{{{\rm I} \kern -.15em {\rm R} }}
\newcommand{\nat}{{{\rm I} \kern -.15em {\rm N} }}

\newcommand{\ud}{{d}}

\def\R{\mathbb R}
\def\N{\mathbb N}

\def\E{\mathbb E}
\def\P{\mathbb P}

\def\cala{{\mathcal A}}
\def\calb{{\mathcal B}}

\def\cald{{\mathcal D}}
\def\cale{{\mathcal E}}
\def\calf{{\mathcal F}}

\def\calk{{\mathcal K}}
\def\call{{\mathcal L}}
\def\calm{{\mathcal M}}

\def\calp{{\mathcal P}}

\def\calu{{\mathcal U}_p}

\def\cals{{\mathcal S}}
\def\calo{{\mathcal O}}

\def\to{\rightarrow}

\def\<{\left\langle }
\def\>{\right\rangle }

\begin{document}

%\title{Verification theorems for stationary stochastic optimal control problems in  Hilbert spaces by means of mild solutions of HJB equations}

\title{Verification theorems for stochastic optimal control problems in Hilbert spaces by means of a generalized Dynkin formula}

\author{Salvatore Federico\footnote{Universit\`a degli Studi di Siena, Dipartimento di Economia Politica e Statistica, Piazza San Francesco 7, 53100, Siena (Italy). Email: \texttt{salvatore.federico@unisi.it}.} \and Fausto Gozzi\footnote{LUISS University, Dipartimento di Economia e Finanza, Viale Romania 32, 00197, Rome (Italy). Email: \texttt{fgozzi@luiss.it}. }}

\maketitle

\begin{abstract}
Verification theorems are key results to successfully employ the dynamic programming approach to optimal control problems. In this paper we introduce a new method to prove verification theorems for infinite dimensional stochastic optimal control problems. The method applies in the case of additively controlled Ornstein-Uhlenbeck processes, when the associated Hamilton-Jacobi-Bellman (HJB) equation admits a \emph{mild solution} (in the sense of \cite{FedericoGozzi16}).
The main methodological novelty of our result relies on the fact that it is not needed to prove, as in previous literature (see e.g. \cite{Gozzi95}), that the mild solution is  a \emph{strong solution}, i.e. a suitable limit of classical solutions of the HJB equation. To achieve the goal we prove a new type of Dynkin formula, which is the key tool  for the proof of our main result.

\bigskip

\vskip 0.15cm

\vskip 0.15cm

\vskip 0.30cm

\textbf{Acknowledgements.} The authors are sincerely grateful to Franco Flandoli, Ben Goldys and Mauro Rosestolato for fruitful discussions on {Subsection \ref{FEEDBACKNEUMANN} and Remark \ref{rem:peano}.}
{The authors are also grateful to an anonymous referee for careful scrutiny and useful suggestions that led to an improved version of the paper.}
\end{abstract}

\setcounter{tocdepth}{3}

\tableofcontents

\section{Introduction}
\label{sec:2015-05-01:00}

In this paper we introduce a new technique, based on a generalized Dynkin formula, to prove verification theorems for stochastic optimal control problems over infinite horizon in Hilbert spaces.

Verification theorems are  key results to enable to solve in a closed way optimal control problems through the dynamic programming approach. Once a solution (in some sense to be precised) of the associated HJB equation is known to exists, the verification theorem provides a sufficient (sometimes also necessary) condition of optimality, which can be used to find optimal controls in feedback forms through the so called closed loop equation.
In the stochastic case, when the solution $v$ is sufficiently smooth, the proof of such theorem is substantially based on an applying the Dynkin formula to the function $v$ and to the state process. In  our framework of discounted time-homogeneous infinite horizon problems the dependence on time is known, so the HJB equation is elliptic and $v$ only depends on the state variable. Hence,  in the finite dimensional case, to employ the classical Dynkin formula, it is needed to know that $v\in C^2$.
 Fortunately, in the finite dimensional case, due to the presence of a powerful regularity theory (at least for nondegenerate second order HJB equations) there is a wide class of problems for which actually $v$ is known to enjoy this regularity, hence the classical Dynkin formula applies and the verification theorem can be proved.
On the other hand, if $v$ is not known to be sufficiently smooth (i.e. when $v$ is known to be only a viscosity solution), still in the finite dimensional case, other techniques have been developed to overcome the fact that the classical Dynkin formula is not applicable.
We mention the following techniques.
\begin{itemize}
  \item[-] The technique developed in \cite{GozziSwiechZhou05}, dealing with  \emph{viscosity solutions}. In this case, the classical Dynkin formula is applied to test functions and only some weak results are obtained.
  \item[-] The technique developed in \cite{PardouxPeng92}. Here  a solution  $v\in C^1$ is obtained through the solution of a suitable backward SDE (BSDE). This technique applies to semilinear HJB equations and provides the verification theorem as a byproduct of the construction itself of the solution $v$. The latter feature is particularly meaningful, as it allows to completely bypass the problem of second order regularity of $v$ and the application of the classical Dynkin formula. On the other hand, the powerfulness of this approach is partly limited by the fact that it can be applied only when a \emph{structural condition} is verified by the control operator.
  \item[-] The technique developed in \cite{GozziRusso06a}: here $v$ is studied and treated as a \emph{strong solution}, i.e. as a suitable limit of classical solutions.
\end{itemize}

When the state space $H$ is infinite dimensional the situation is much worse. First of all, the regularity needed to apply the classical Dynkin formula (see, e.g.,  \cite[Sec.\,4.4]{DaPratoZabczyk14}) is very demanding and does not allow to deal with many applied examples proposed and only partly studied in the literature. This is partly due to additional regularity assumptions on the coefficients needed in infinite dimension, partly due to the lack of a satisfactory regularity theory in infinite dimension.
Hence, elaborating alternative methods is considerably more important than in the finite dimensional case.
Clearly, the first attempt consists  in trying to extend the techniques developed in the finite dimensional case to infinite dimensional one. On this side, so far the state of the art  can be  basically depicted as follows.
\begin{itemize}
  \item[(a)] There are no results concerning the case when $v$ is a viscosity solution.
  \item[(b)] Results with the BSDE approach have been elaborated in various papers, see e.g. \cite{FuTe-ell} in the infinite horizon case, but always under the structural condition. The latter requirement leaves out the treatments of important cases like boundary control of stochastic PDEs or delayed control of SDEs.
  \item[(c)] Results dealing with strong solutions  are available in \cite{GozziRouy96} and in \cite{Cerrai01-40}.
\end{itemize}
The results we provide here are closer, in the conclusions, to the results mentioned in item (c) above.  With respect to them, ours have a larger range of applicability and, not only in this sense, can be seen as a significant improvement of this technique, as  we will comment more precisely afterwards.

We stress the fact that our method to prove the verification theorem is a novelty also in finite dimension: our results may be useful to treat also finite dimensional problems where only partial regularity properties of the value function are known. Here we focus on the infinite dimensional case where the application is more meaningful.

We now illustrate the results and the novelties of our paper.
We consider a class of stochastic optimal control problems in a real separable Hilbert space $H$, where the noise is additive and the control only appears in an additive form in the drift term. More precisely, the state equation is
\begin{equation}\label{OUintro}
dX(t)=\big[AX(t)  + GL(u(t))\big]\,dt
+ \sigma \,dW(t),
\end{equation}
where  $A:{\cal{D}}(A)\subseteq H\to H$, $G:K\to H$, $L:\Lambda \to K$, $\sigma:\Xi\to H$ are suitable operators, with $K,\Xi$ being other real separable Hilbert spaces  and $\Lambda$ being a Polish space; $W$ is a $\Xi$-valued cylindrical Browian motion; $u$ is the control process taking values in $\Lambda$;   $X$ is the state process taking values in the Hilbert space $H$. The stochastic control problem consists in minimizing, over a set of admissible control processes, a cost functional in the form
\begin{equation*}
{\E}\left[ \int_{0}^{\infty}e^{-\lambda s}l\big(X(s),u(s)\big)\,ds\right],
\end{equation*}
where $\lambda>0$ is a discount factor and $l$ is a suitable real valued function.
In this case the associated HJB equation is an elliptic semilinear PDE in the space $H$:
$$
 \lambda v(x) -\frac{1}{2}\;
\mbox{\rm Tr}\;[\sigma\sigma^*D^2v(x)] -\< Ax,Dv(x)\>_H- F_0(x,D^Gv(x))=0,
$$
where 
$$
	F_0(x,D^Gv(x)) =
	 \inf_{u \in \Lambda }
\left\{  \< L(u),D^Gv(x) \>_K +l(x,u)\right\},
$$
where $D^Gv$ denotes  the $G$-gradient of a function $v:H\to\R$ (see Subsection \ref{SS4:GDER}). 
Under reasonable assumptions, it is proved in \cite{FedericoGozzi16} that such HJB equation admits a unique mild solution, i.e. a solution of a suitable integral form of the above equation. Such solution admits $G$-gradient, i.e. verifies the minimal differentiability requirement to give   sense to the nonlinear Hamiltonian term $F_0$ in HJB above.
Once one proves the existence of a mild solution $v$ to the associated HJB equation, the approach of item (c)  would require three nontrivial technical steps: first, proving that such a mild solution is indeed a strong solution (limit, in a suitable sense, of classical solution); second, applying Dynkin formula to the approximating classical solutions; third, passing to the limit the Dynkin formula.
As one may expect, passing through all these steps   requires additional hypotheses that may be nontrivial to check  in practice (see e.g. \cite{GozziRouy96}). %and \cite[Section 4.5]{FGSbook}).
Our goal here is to bypass these steps through an alternative path.
 In fact, we show that the role of strong solutions is not essential. Indeed, relying on the theory of $\pi$-semigroups (see e.g. \cite[Appendix B]{FGSbook} and \cite{Priola99}),
we prove a generalized (abstract) Dynkin formula --- deserving interest in itself ---  which can be directly applied to mild solutions. The proof is quite involved and this is the reason why we consider here the case of stochastic control of equation of type \eqref{OUintro}, where the uncontrolled part of the state equation is of Ornstein-Uhlenbeck type\footnote{It is worth to stress that, even if in the case of Ornstein-Uhlenbeck dynamics the approach of strong solutions has already been succesfully applied (see \cite{GozziRouy96}), the method used here, other than being  original, seems to be extendable to more general structures of state equations, where the strong solution approach would fail.}. Then, relying on this formula, we straightly prove a verification theorem.
 The new results on $G$-derivatives provided in \cite{FedericoGozzi16} (see also \cite[Ch.\,4]{FGSbook}) enable us to  apply our method to more general examples than the ones treated by the current literature; in particular, to cases where the structural condition required at item (b) above is not verified (see Section \ref{sec:app}).

The main results of the paper are the abstract Dynkin formula (Theorem \ref{teo:dynkin}); the  verification theorem (Theorem \ref{th:ver});
the consequent Corollary \ref{cr4:optimalfeedbackell} on sufficient conditions for the existence of optimal control processes in feedback form.
Moreover, since the existence of optimal feedback controls might be is easier to obtain
when the optimal control problem is considered in the weak formulation, i.e., letting also the stochastic basis to vary, we also provide Corollary \ref{cr4:optimalfeedbackellbis} in this direction.
%\blu{We underline that,} concerning the existence of \blu{feedback} control processes, we do not provide general results on existence of them,
We underline that we do not provide general results on
the existence of optimal control processes in feedback form,
as such results strongly depend on the specific case at hand.  To this regard, in Section \ref{sec:app} --- where we deal with two specific applications: optimal boundary control (of Neumann type) of the stochastic heat equation and optimal control of SDEs with delay in the control variable ---  we provide for the first example some results and comments on the existence of optimal feedback control processes.

The paper is organized as follows.
After some preliminaries in Section \ref{SS:pre} on spaces, notation and the notion of $G$-derivative recently extended in \cite{FedericoGozzi16}, we introduce our family of control problems in Section \ref{SS:SOC}.
Section \ref{SS:DYN} is devoted to prove our new Dynkin formula (Theorem \ref{teo:dynkin}), the methodological core of the paper. In Section \ref{SS:HJB} we prove our main results on the control problem: in Subsection \ref{SUB:VER}, the verification theorem (Theorem \ref{th:ver}); in Subsection \ref{SS:FEEDBACKS},
 Corollary \ref{cr4:optimalfeedbackell} on optimal feedbacks.
Section \ref{sec:app} is devoted to illustrate the applications of our results to the aforementioned examples. {Finally the Appendix is devoted to prove few technical results needed to prove our Dynkin formula.}

%\begin{equation}
%\label{eq4:HJbasicelliptic1}
%\lambda v(x) -\frac{1}{2}\;\mbox{\rm Tr}\;[Q(x)D^2v(x)]
%-\langle Ax+ b(x),Dv(x) \rangle  -F_0(x,v(x),D^{G}v(x))=0,
%\quad  x \in H.
%\end{equation}
\section{Preliminaries}\label{SS:pre}
In this section we provide some preliminaries about spaces and notation used in the rest of the paper and recall from \cite{FedericoGozzi16} the notion of $G$-derivative. We restrict the treatment of $G$-derivative to the case of real valued functions defined on Hilbert spaces and to constant operator maps $G$. This will be enough for the purposes of the present paper. For a more general theory and more details we refer to the aforementioned paper \cite{FedericoGozzi16}.
\subsection{Spaces and notation}
\paragraph*{Measurable bounded and continuous functions.} All the topological spaces are intended endowed with their Borel $\sigma$-algebra, denoted by $\calb$. By {measurable} set (function), we always intend a {\emph{Borel}} measurable set (function). If $U$ is a topological space and $V$ is a topological vector space, we denote by $B_b(U,V)$ the set of bounded measurable functions from $U$ to $V$ and by $C_b(U,V)$ the set of bounded continuous functions from $U$ to $V$. If $V=\R$, we drop it in the latter notation. If $V$ is complete,
the spaces $B_b(U,V)$ and $C_b({U},V)$   are  Banach spaces when endowed
with the norm
\begin{equation}\label{supnorm}
|\varphi |_{{\infty}} =  \sup_{x \in {U}}|\varphi (x)|_{V}.
\end{equation}

\paragraph*{Hilbert spaces.} Let $H$ be a Hilbert space. We denote its
norm by $|\cdot|_H$ and its inner product by
by $\left \langle \cdot,\cdot\right\rangle_H$.
We  omit the subscript  if the context is clear and if $H=\R$.
If a sequence $(x_n)_{n\in\mathbb{N}}\subseteq H$,   converges to $x\in U$ in the norm (strong) topology we write $x_n\to x$.

We denote by $H^*$
the topological dual of $H$, i.e. the space of all continuous
linear functionals defined on $H$.
We always identify  $H^*$ with $H$ through the standard Riesz identification.

%If $U$ is a Hilbert space and $\left\{ e_{k} \right\}_{k \in \mathbb{N}}$ is an orthonormal basis of $U$, we use, for $x \in U$, the notation $x_{k}:= \langle x,e_{k}\rangle$.
%

 \paragraph*{Linear operators.}
Let $H,K$ be {real separable} Hilbert spaces. We denote by
$\mathcal{L}(H,K)$ the set of all bounded (continuous) linear operators
$T :  H \to K$ with norm $|T|_{\mathcal{L}(H,K)}:= \sup_{x \in H, x \ne 0} \frac{|Tx|_K}{|x|_H}$, using for simplicity the notation $\mathcal{L}(H)$ when $H=K$.
Moreover, we denote by $\call_u (H,K)$ the space of closed densely defined and possibly unbounded linear operators $T:\mathcal{D}(T)\subseteq H \to K$, where $\cald(T)$ denotes the domain. We recall that $\cald(T)$ is a Hilbert space when endowed with the graph norm $|x|_{\cald(T)}=|x|_H+|Tx|_K$. The range of an operator $T\in\call_u(H,K)$ is denoted by $\calr(T)$.
 Clearly, $\call (H,K)\subseteq \call_u(H,K)$. Given  $T\in \call_u (H,K)$,  we  denote its adjoint operator by $T^*:\mathcal{D}(T^*)\subseteq K \to H$.

 We denote by   $\call_1(H)$ the set of trace class operators, i.e. the operators $T\in\call(H)$ such that, given an orthonormal basis  $\{e_k\}_{k\in\N}$  of $H$, the quantity
 $$|T|_{\call_1(H)}\coloneqq\sum_{k=1}^\infty \big\langle (T^*T)^{1/2} e_k,e_k\big\rangle_H$$
 is finite (see \cite[Sec.\,VI.6]{ReedSimon80}).  The latter quantity is independent of the basis chosen and defines a norm making $\call_1(H)$ a separable Banach space.
  The trace of an operator $T\in\call_1(H)$ is denoted by $\mbox{Tr}[T]$, i.e. $\mbox{Tr}[T]\coloneqq \sum_{k=0}^\infty \langle T e_k,e_k\rangle_U$. The latter quantity is finite and, again,  independent of the basis chosen.
 We denote  by  $\call_1^+(U)$ the subset of $\call_1(H)$ of self-adjoint nonnegative  (trace class) operators on $H$. Note that, if $T\in\call_1^+(H)$, then $\mbox{Tr}[T]=|T|_{\call_1(U)}$.

We denote by $\call_2(H,K)$ (subset of $\call(H,K)$) the space of Hilbert-Schmidt operators from $H$ to $K$, i.e the spaces of operators such that, given an orthonormal basis  $\{e_k\}_{k\in\N}$  of $H$, the quantity
 $$\big|T\big|_{\call_2(H)}\coloneqq\left(\sum_{k=0}^\infty \big|Te_k\big|_K^2\right)^{1/2}$$
 is finite (see \cite[Sec.\,VI.6]{ReedSimon80}).  The latter quantity is independent of the basis chosen and defines a norm making $\call_2(H)$ a Banach space.  It is actually a Hilbert space with the scalar product
$$
\big\langle T, S\big\rangle_{\call_2(H,K)}\coloneqq \sum_{k=0}^\infty \big\langle Te_k, Se_k\big\rangle_K,$$
where $\{e_k\}_{k\in\N}$ is any orthonormal basis of  $H$.

\paragraph*{Stochastic processes.} Let $(\Omega, \mathcal{F}, (\mathcal{F}_t)_{t\geq 0}, \P)$ be a filtered probability space satisfying the usual conditions.
Given $p\in[1,+\infty)$, $T>0$, and a Hilbert space $U$, we denote by $\mathcal{M}_\calp^{p,T}(U)$ the set of all (equivalence classes of) progressively measurable processes
$X\colon [0,T] \times \Omega\to U$ such that
\[
\big|X\big|_{\mathcal{M}_\calp^{p,T}(U)} := \left (\int_0^T\mathbb{E} \left[|X(s)|_U^p\right] ds\right )^{1/p} < \infty.
\]
This is a Banach space with the norm $|\cdot|_{\mathcal{M}_\calp^{p,T}(U)}$.
Next, we denote by $\mathcal{M}_\calp^{p,loc}(U)$ the space of all (equivalence classes of) progressively measurable processes
$X\in \mathcal{M}_\calp^{p,T}(U)$ such that
$X|_{[0,T]\times\Omega}\in \mathcal{M}_\calp^{p,T}(U)
$
for every $T>0$.
{We denote by $\mathcal{K}_\calp^{p,T}(U)$ the set of all (equivalence classes of) progressively measurable processes
$X\in \mathcal{M}_\calp^{p,T}(U)$ such that
$$
[0,T]\to L^p(\Omega,U), \ \ t\mapsto X(t)
$$
is continuous.
This is a Banach space with the norm
\[
\big|X\big|_{\mathcal{K}_\calp^{p,T}(U)} := \sup_{s\in[0,T]} \left(\mathbb{E} |X(s)|_U^p \right )^{1/p}.
\]
Next, we denote by $\mathcal{K}_\calp^{p,loc}(U)$ the space of all (equivalence classes of) progressively measurable processes
$X\colon [0,+\infty) \times \Omega\to U$ such that
$X|_{[0,T]\times\Omega}\in \mathcal{K}_\calp^{p,T}(U)
$
for every $T>0$.
We also say that
{elements of $\mathcal{K}_\calp^{p,T}(U)$ and $\mathcal{K}_\calp^{p,loc}(U)$
are ``$p$-mean continuous''.}}
%\red{Given $T>0$ and a Hilbert space $U$, we denote by
%$\mathcal{K}^{p,T}_{\mathcal{P}}(U)$ the space
% of all (equivalence classes of) progressively measurable processes
%$X\colon [0,T] \times \Omega\to U$ admitting a version with continuous trajectories and  such that
%\[
%|X(\cdot)|_{\mathcal{K}_\calp^{2,T}(U)} := \E \left[\sup_{s\in [0,T]} |X(s)|_U^2 \right ]^{1/2} < +\infty.
%\]}
%This is a Banach space with the norm $|\cdot|_{\mathcal{K}_\calp^{p,T}(U)}$.
%We denote by $\mathcal{K}_\calp^{p,loc}(U)$ the set of all (equivalences classes of) progressively measurable processes
%$X\colon [0,+\infty) \times \Omega\to U$ such that
%$X|_{[0,T]\times\Omega}\in \mathcal{K}_\calp^{p,T}(U)
%$
%for every $T>0$.
%
%\subsection{Bochner integration}
%Let $I\subseteq\R$, let $V$ be a Banach space, and let $f:I\to V$ be measurable.
%We recall that, if $|f|_V\in L^1(I,\R)$, then $f$ is Bochner integrable, and we write $f\in L^1(I,V)$. Moreover, in this case
%\begin{equation}\label{ppppk}
%\bigg\langle v^*,\  \int_I f(t)dt\bigg\rangle_{V^*,V} =\ \int_I\langle v^*,f(t)\rangle_{V^*,V}dt, \ \ \ \forall v^*\in V.
%\end{equation}
%Finally,  we recall that, if $V$ is separable, by Pettis measurablility Theorem \cite[Th.\,1.1]{Pettis38}, $f$ is measurable if and only if  $t\mapsto \langle v^*, f(t)\rangle_{V^*,V}$ is measurable for every $v^*\in V^*$.

\subsection{$G$-derivative}
\label{SS4:GDER}
Here we provide the notion of
$G$-derivative for functions $f:{H}\rightarrow \R$, where $H$ is a Hilbert space. The latter notion is considered in  \cite{FedericoGozzi16} when $G$ is a map $G:{U}\to \call_u(Z,U)$, with $U,Z$ Banach spaces. Here we restrict to the case of constant $G$.

Recall that, if $f:H\to \R$, the Fr\'echet derivative  of $f$ at $x$ (if it exists) is the (unique) linear functional $Df(x)\in  H^*\cong H$ such that
  $$\lim_{|h|_H\to 0} \frac{\big|f\left(  x+h\right)  -f\left(  x\right)- \langle Df(x),h \rangle_H\big|}{|h|_H}=0.
$$

\begin{Definition}[$G$-derivative]
\label{df4:Gderunbounded}
Let $H,K$  be Hilbert spaces, let $f:{H} \rightarrow \R$ and
 $G\in \call_u (K,H)$.
We say that $f$ is continuously $G$-Fr\'echet differentiable at $x\in H$ (briefly, $G$-differentiable at $x\in H$) if there exists $D^Gf(x)\in K^*\cong K$ (clearly, if it exists, then it is unique), called the $G$-derivative of $f$ at $x$, such that
\begin{equation}\lim_{k\in \cald(G), \,|k|_K\to 0} \frac{\big|f\left(  x+Gk\right)  -f\left(  x\right)- \langle D^Gf(x),k \rangle_K\big|}{|k|_K}=0.
\label{eq4:Gder}
\end{equation}
%\begin{equation}
%D^{G}f\left(  x;h\right)  =\lim_{s\rightarrow 0}\frac{f\left(  x+sG\left(
%x\right)  h\right)  -f\left(  x\right)  }{s},\text{ }s\in\mathbb{R},
%\label{eq4:Gder}
%\end{equation}
%where the previous limit is to be taken in the norm of $V$.
\end{Definition}
We denote by $C^{1,G}_b(H)$  the space of all maps $f:H\to \R$ such that $f$ is continuously $G$-differentiable over $H$, i.e. such that $f$ is $G$-differentiable at each $x\in H$ and the map $D^Gf:H\to K$  belongs to $C_b(H,K)$. In the special case $K=H$ and $G=I$, we simply use the standard notation   $C^1_b(H)$.
\begin{Remark}\label{rmGder}
Note that, in the definition of the $G$-derivative, one considers only the directions
in $H$ selected by the range of $G$.
When $K=H$ and $G=I$ it reduces to the Fr\'echet derivative, i.e. $Df=D^Gf$.
Clearly, if $f$ is $G$-differentiable at $x$, then it is also $G$-Gateaux differentiable at $x$, in the sense that
\begin{equation}\label{G-direc}
\lim_{t\to 0} \frac{f(x+tGk)-f(x)}{t}=\big\langle D^Gf(x),k\big\rangle_K,  \		\ \ \ \forall k\in\cald(G);
\end{equation}
moreover, the limit above is uniform in $k\in \cald(G)\cap B_K(0,R),$ for every $R>0$. Conversely, if there exists $k'\in K$ such that
\begin{equation}\label{G-direc2}
\lim_{t\to 0} \frac{f(x+tGk)-f(x)}{t}=\langle k',k\rangle_K, \ \ \mbox{uniformly in} \ k\in \cald(G)\cap B_K(0,R), \ \forall R>0,
\end{equation}
 then $f$ is $G$-differentiable at $x\in H$ and $D^Gf(x)=k'$.
\end{Remark}
 The notion of  $G$-derivative
allows  to deal with functions which are not Gateaux differentiable, as shown by the following example.
\begin{Example}\label{ex4:Gnondiff}

Let $f:\R^{2}\rightarrow \R$ be defined by $f(x_1,x_2)\coloneqq\left\vert
x_{1}\right\vert x_{2}$. Clearly, $f$ does not admit   directional derivative in the direction $(1,0)$ at the point $(x_1,x_2)=(0,1)$. On the other hand,
if we consider $G\in \call (\R^2)\cong  \R^2$, defined by  $G=(0,1)$, then
 $f$ admits
$G$-Fr\'echet derivative at every $(x_1,x_2) \in \R^2$.
\end{Example}
\begin{Remark}\label{rem:mmn}
Clearly, if $f$ is  Fr\'echet differentiable at some $x\in H$  and $G\in \call(K,H)$, it turns out that $f$ is $G$-Fr\'echet differentiable at $x$ and
\begin{equation}\label{eq4:DGDfG}
    D^{G}f\left(x\right)  =G^*D f\left(  x\right).
\end{equation}
Also, if $f$ is  both Fr\'echet differentiable and $G$-differentiable at some $x\in H$, then $Df(x)\in \cald(G^*)$ and
\eqref{eq4:DGDfG} holds true. Indeed,
%denoting by $G^{-1}:H\to \cald(G)$  the pseudo-inverse of $G$\footnote{}, taking $h\in H$ and setting $k=G^{-1} h$,
we get by Fr\'echet differentiability
  $$\lim_{s\to 0} \frac{f\left(  x+sGk\right)  -f\left(  x\right)}{s}= \big\langle Df(x),G k\big\rangle_H, \quad \forall k\in \cald(G).
$$
On the other hand, by $G$-Fr\'echet differentiability we also have
  $$\lim_{s\to 0} \frac{f\left(  x+sGk\right)  -f\left(  x\right)}{s}= \big\langle D^Gf(x),k\big\rangle_K,\quad \forall k\in \cald(G).
$$
Hence
$$
\big|\<Df\left(  x\right),Gk\>_H\big|= \big|\big\langle D^G f(x),k\big\rangle_K \big|\le \big|D^G f(x)\,\big|_K|k|_K,\quad \forall k\in \cald(G).
$$
It follows what claimed.
\end{Remark}
%\begin{Remark}\label{rm:GstarD}
% {\blu{The contents of Remark \ref{rem:mmn} can be made sharper.
%If $f$ is  Fr\'echet differentiable at some $x\in H$, $G\in \call_u(K,H)$
%is unbounded, and $Df(x)\in \cald(G)$, then  $f$ is $G$-G\^ateaux differentiable and, denoting still by $D^G f(x)$ its $G$-G\^ateaux derivative, we  have
%$D^{G}f\left(x\right)  =G^*D f\left(  x\right)$ (see Remark \ref{rmGder}). If, moreover, this happens for each $x\in H$ and the map $x\mapsto G^*Df(x)$ is continuous, then, recalling  \cite[Remark 4.19]{FGSbook}, one gets that  $f$ is continuously $G$-Fr\'echet differentiable.}}
%\end{Remark}

If $G$ is  unbounded, a function $f:H\to\R$ may be Fr\'echet-differentiable at some $x\in H$ and yet not $G$-Fr\'echet differentiable there, as shown by the following example.

\begin{Example}
%When $G$ only takes values in $\call_u(U,Z)$, that is when we deal with the possibility that $G(x)$ is unbounded, then
%even if $f$ is Fr\'echet differentiable at all points $x\in {U}$,
%the $G$-Fr\'echet derivative may not exist in some points.
%Indeed, consider the following example.
Let $H, K$ be  Hilbert spaces, let
$G:\cald(G)\subsetneq K \to H$ be a closed densely defined  unbounded linear operator on $H$, and let $G^*:\cald(G^*)\subsetneq H\to K$ be its adjoint. Next,
let $f:U \to \R$ be defined by $f(x)\coloneqq \frac{1}{2}|x|_H^2$. Clearly, $f$ is Fr\'echet differentiable  at every $x \in H$ and $Df(x)= x$.
On the other hand, if $f$ was also $G$-differentiable at every  $x\in H$, by Remark \ref{rem:mmn}
it would follow $x\in \cald(G^*)$ for every $x\in H$, i.e. $\cald(G^*)=H$, a contradiction.
\end{Example}

\color{black}

\section{Formulation of the stochastic optimal control problem}\label{SS:SOC}
We are concerned with the optimal control of an Ornstein-Uhlenbeck process valued in a Hilbert space $H$. Precisely, let $H,K,\Xi$ three real separable Hilbert spaces, {let $(U,|\cdot|_{U})$ be a real Banach space and let $\Lambda\subseteq U$ be measurable and endowed with the $\sigma$-algebra induced by $\calb(U)$, the Borel $\sigma$-algebra of $U$.}
Let $(\Omega,\mathcal{F},\{\mathcal{F}_t\}_{t\geq 0}, \P)$ be a {complete} filtered probability space {satisfying the usual conditions}, let $W=(W_t)_{t\geq 0}$ be a $\Xi$-valued cylindrical Brownian motion (see \cite[Ch.\,4]{DaPratoZabczyk14}),
and consider the controlled SDE
\begin{equation}
\begin{cases}
dX(t)=\big[AX(t) + GL(u(t))\big]\,dt
+ \sigma\, dW(t), \ \ \ t\geq 0,\\
\label{eq4:SE1ellbis}
X(0)=x,
\end{cases}
\end{equation}
where the control process $u(\cdot)$, taking values in $\Lambda$, belongs to a suitable space of admissible controls
and the coefficients $A,G,L, \sigma$ satisfy the following assumptions, which will be standing and not repeated throughout the paper.
\begin{Assumption}
\label{hp4:ABQforOU}
\begin{enumerate}[(i)]
\item[]
\item $A:\cald(A)\subseteq H\to H$ is a closed densely defined linear operator generating  a $C_0$-semigroup $\big\{e^{t A}\big\}_{t\geq 0}$ of  operators of $\mathcal{L}(H)$.

\item
 $\sigma\in \call(\Xi,H)$, \ $e^{sA}\sigma\sigma^{\ast}e^{sA^{\ast}}\in \call_1(H)$ for all $s >0$,
and
there exists $\gamma \in (0,1/2)$ such that
$$
\int_{0}^{t} s^{-2\gamma}{\rm Tr}\left[ e^{sA}\sigma\sigma^{\ast}e^{sA^{\ast}}\right]ds
<\infty \ \ \ \forall t\geq 0.
$$
\item $G:\cald(G)\subseteq K\to H$ is a closed densely defined\footnote{The assumption that $G$ is densely defined can be done without loss of generality, as one can always restrict $K$ to $\overline{\cald(G)}$.}  linear operator  such that
$e^{sA}G:\cald(G)\to H$ can be extended for every $s>0$ to a continuous linear operator defined on $K$ that we  denote by $\overline{e^{sA}G}$. Moreover, there exists %$p\in[1,+\infty]$,
$C_G>0$, $a_G \in \R$ and $\beta \in [0,1)$
%$f_G\in L^p_{loc}([0,+\infty),\R)$
such that
\begin{equation}\label{esg}
 \left|\overline{e^{sA}G}\right|_{\call(K,H)}\le \ C_G(s^{-\beta}\vee 1)\,e^{a_Gs} \ \ \forall s>0.
\end{equation}
\item  $L:\Lambda\to K$ is measurable and $\big|L(u)\,\big|_K\leq C_L(1+|u|_{U})$ for some $C_L>0$.

%\item\label{GL} $L\in B_b(U;K)$ and
%   there  exists  $f_{GL}\in L^1_{loc}([0,+\infty), \mathbb{R})$ 	and a modulus of continuity $w$ such that
%%\begin{equation}\label{primaipotesisug1}
%%|\overline{e^{sA}G}L(u)|\leq  f(s)
%%  (1+|x|), \ \ \ \forall s> 0, \ \forall x\in H,
%%  \end{equation}
%%  and
%  \begin{equation}
%  |\overline{e^{sA}G}(L(u)-L(u'))|\leq
%  f(s)w(d(u,u')), \ \ \ \ \forall s> 0, \ \forall u,u'\in U.
%  \label{primaipotesisug2}
%\end{equation}
\end{enumerate}
\end{Assumption}
\begin{Remark}\label{Rem:ext}
Since for every $t>0$ and $s\geq 0$ the operators
 $\overline{e^{(s+t)A}G}$ and $e^{sA}\overline{e^{tA}G}$ belong to $\call(K,H)$ and coincide on the dense subset $\cald(G)\subseteq K$, we have
\begin{equation}\label{eq:semi1}
\overline{e^{(s+t)A}G} = e^{sA}\overline{e^{tA}G}, \ \ \forall t>0,\ \forall s\geq 0.
\end{equation}
{This {implies} that the map $(0,+\infty)\to \mathcal{L}(K,H), \  s\mapsto \overline{e^{sA}G}$ is strongly continuous, i.e.
$s\mapsto \overline{e^{sA}G}x$ is continuous for each $x\in H$.}
\end{Remark}
%Let $p'$ the conjugate exponent of $p\in[0,+\infty]$ of Assumption \ref{hp4:ABQforOU}(iii).
We now take
\begin{equation}\label{eq:choicep}
p \in \left(\frac{1}{1-\beta}, +\infty\right),
\end{equation}
which will be fixed in the rest of the paper.
We consider, as space of admissible controls, the space of processes
\begin{equation}\label{up}
\calu:=\left\{u:\Omega\times[0,+\infty)\to \Lambda \ \mbox{prog. meas. and s.t.} \
\int_0^t \E\left[|u(s)|_U^{p}\right]\,ds <\infty \ \ \  \forall t\geq0 \right\}.
\end{equation}
The reason for the choice of $\beta$ in \eqref{esg} and of $p$ in \eqref{eq:choicep}-\eqref{up} relies on the following result ({cf. also \cite[Prop.\,8.8]{Folland99} and \cite[Lemma\,3.2]{GawareckiMandrekar10}}), which will guarantee well-posedness of  the controlled state equation (Proposition \ref{prop:SE}).
\begin{Lemma}\label{lemmaDP}
{Let $E,V$ be real  Banach spaces, let $\beta\in[0,1)$, $p>\frac{1}{1-\beta}$. Let $f\in L_{loc}^p([0,+\infty);E)$ and let $g:(0,+\infty)\to \mathcal{L}(E,V)$ be strongly continuous\footnote{Meaning that
$g(\cdot)e:(0,+\infty)\to V$ is continuous for each $e\in E$.}
and such that $|g(s)|_{\mathcal{L}(E,V)}\leq {C_0(s^{-\beta}\vee 1)}$ for some $C_0>0$ for every $s\in(0,+\infty)$. Then $F:\R^+\to V$ defined as Bochner integral} by
$$
F(t):= \int_0^t g(t-s)f(s)\,\,ds, \ \ t\in \R^+,
$$
is {well defined} and continuous.
\end{Lemma}
\begin{proof}
{Let $t>0$.} {First of all, we note that the map
$$
[0,t)\to V, \ \  s\mapsto g(t-s)f(s),
$$
is measurable for each $t>0$.} Indeed, given $t>0$ the above map can be seen as the composition $h_1\circ h_2$ where
$$
h_1:(0,t]\times E \to V,\quad h_1(s,e)=g(s)e;
\qquad h_2:[0,t) \to (0,t]\times E, \quad h_2(s)=(t-s,f(s)).
$$
Now, $h_2$ is clearly measurable. Also  $h_1$ is measurable, as it is continuous:  {indeed $g(\cdot)e$ is continuous for each $e\in E$ and $\{g(s)\}_{s\in[\varepsilon,t]}\subseteq \mathcal{L}(E,V)$ is a family of uniformly bounded operators for each $\varepsilon\in(0,t)$.} Hence $h_1\circ h_2$ is measurable.

Given the above, it makes sense to consider $ \int_0^t g(t-s)f(s)ds$ {in Bochner sense} for each $t>0$.
%Moreover
%$$
%\int_0^t |g(t-s)f(s)|_Vds\leq \int_0^t (t-s)^{-\beta} |f(s)|_Vds<+\infty
%$$
%for every $t\in\R^+$. This shows that $F$ is well defined as Bochner integral in $V$.
By H\"older's inequality, setting $\kappa:=-\frac{\beta p}{p-1}+1>0$, we have for each $t>0$
$$
\int_0^t |g(t-s)f(s)|_Vds\leq \int_0^t (t-s)^{-\beta} |f(s)|_Vds 	\leq \left(\int_0^t (t-s)^{-\beta\frac{p}{p-1}}ds\right)^{\frac{p-1}{p}}|f|_{L^{p}([0,T];\R)} =  \left(\frac{t^{\kappa}}{\kappa}\right)^{\frac{p-1}{p}}|f|_{L^{p}([0,T];\R)}.
$$
This show, at once, that $F$ is well defined as Bochner integral in $V$ and that  $\lim_{t\to 0^+} F(t)=0$, so $F$ is continuous at $0$.

Let us show now that $F$ is continuous on each interval of the form $\left[{t_0}, T\right]$ with $t_0\in(0,T)$. Set, for $\varepsilon\in (0,{t_0})$,
$$
F_\varepsilon(t):=\int_0^{t-\varepsilon}  g(t-s) f(s)ds, \ \ t\in[t_0,T].
$$
By dominated convergence we easily see that $F_\varepsilon$ is continuous on $\left[{t_0},T\right]$. Moreover, using again H\"older's inequality we have, for
%suitable $C_T>0$ and
all $t\in[t_0,T]$
\begin{align*}
|F(t)-F_\varepsilon(t)| \leq
\left(\int_{t-\varepsilon}^t (t-s)^{-\beta\frac{p}{p-1}}ds\right)^{\frac{p-1}{p}}|f|_{L^{p}([0,T];\R)}
= \left(\frac{\varepsilon^{\kappa}}{\kappa}\right)^{\frac{p-1}{p}}|f|_{L^{p}([0,T];\R)}.
\end{align*}
This show  $F_\varepsilon\to F$ uniformly in $[t_0,T]$, hence $F$ is continuous in $[t_0,T]$, concluding the proof.
\end{proof}
\begin{Proposition}\label{prop:SE}
 For each $u(\cdot)\in \calu$, the {process}
\begin{equation}\label{OU}
X(t;x,u(\cdot)):=e^{tA}x+ \int_0^t e^{(t-s)A}\sigma dW(s)
+\int_0^t \overline{e^{(t-s)A} G}L(u(s))ds,
\end{equation}
{is well-defined and belongs to $\calk_\calp^{1,loc}(H)$}. Moreover, it  admits a version with continuous trajectories.
\end{Proposition}
\begin{proof}
{By Remark \ref{Rem:ext} and  Assumption \ref{hp4:ABQforOU}(iii)-(iv), we can apply  Lemma \ref{lemmaDP} with
$$E =L^1(\Omega;K), \ \ V=L^1(\Omega;H), \ \   f(s)= L(u(s))$$
and  $g(s)\in\mathcal{L}(E,V)$ defined by $$ \ \  \big[g(s)Z\big](\omega):= \overline{e^{sA}G}Z(\omega), \ \ \  Z\in L^1(\Omega;K).
$$
It follows that
\begin{equation}\label{conv}
t\longmapsto \int_0^t \overline{e^{(t-s)A} G}L(u(s))ds
\end{equation}
is well defined as stochastic process and belongs to $\calk_\calp^{1,loc}(H)$.
We can repeat the argument employed above dealing now with  trajectories.  Fixing $\omega\in \Omega$ and applying  Lemma \ref{lemmaDP} with
$$E=K, \ \ V=H, \ \   f(s):= L(u(s)(\omega)), \ \
 g(s)= \overline{e^{sA}G},
$$
it follows  that the map
$$\R^+\to H, \ \ t\mapsto \int_0^t \overline{e^{(t-s)A} G}L(u(s)(\omega))ds$$
is continuous. The latter integral expression, for varying $\omega\in\Omega$, clearly provides a version of \eqref{conv} with continuous trajectories.}

{
On the other hand, in view of Assumption \ref{hp4:ABQforOU}(ii),  from
 \cite[Th.\,5.2 and Th.\,5.11]{DaPratoZabczyk14} we know that the \emph{stochastic convolution}
 $$W^A(t):=\int_0^t e^{(t-s)A}\sigma dW(s), \ \ \ t\geq 0,$$ is a (well defined) stochastic process belonging to $\calk_\calp^{2,loc}(H)$ and  admitting a version with continuous trajectories, concluding the proof.}
\end{proof}
We refer to the  process \eqref{OU} as the \emph{controlled Ornstein-Uhlenbeck process} or \emph{mild solution} of SDE \myref{eq4:SE1ellbis}. We always consider its version (unique, up to indistinguishability) with continuous trajectories.

Let $\lambda>0$, $x\in H$, and let $l:H\times \Lambda\to\R$ be such that
\begin{equation}
l \ \mbox{is measurable and bounded from below}\footnote{Cases where $l$ is not bounded from below can be treated adding suitable growth conditions which depends on the specific problem at hand. We do not do it here for brevity. See also Remark \ref{rm:polgrowth} on this.}.
\label{eq4:hpl}
\end{equation}
Consider the functional
\begin{equation}
 	J(x;u(\cdot)) ={\E}\left[ \int_{0}^{\infty}e^{-\lambda s}l\big(X(s;x,u(\cdot)),u(s)\big)\,ds\right], \ \ \ \ x\in H, \ u(\cdot)\in \calu.
    \label{eq4:CF1ell}
\end{equation}
By  \eqref{eq4:hpl}, the functional  above is well defined {(possibly with value $+\infty$)} for all $x\in H$ and $u(\cdot)\in \calu$.
%\,(\footnote{{It is possible to consider also cases when $l$ is not bounded from below provided suitable conditions are satisfied to keep the value function finite.}}).
% see e.g. \cite{BiffisGozziProsdocimi16}}}).
The stochastic optimal control problem consists in minimizing the functional  over the set of admissible controls $\calu$, i.e. in solving the optimization problem
% (\ref{eq:Vfunction-strong})
%(here, to simplify the notation, we write $V(t,x)$ instead of $V_t^\mu(t,x)$).
\begin{equation}
 		V(x) \coloneqq
 \inf_{u(\cdot)\in \calu } J(x;u(\cdot)), \ \ \ \ x\in H.
 	    \label{eq4:VF1ell}
\end{equation}
The function $V:H\to \R\cup\{+\infty\}$ is the so called value function of the optimization problem. If $x \in H$ is such that $V(x)< \infty$ and $u^*(\cdot)$ is such that $V(x)=J(x;u^*(\cdot))$, then $u^*(\cdot)$ is called {\em optimal strategy} and the associated state trajectory is called {\em optimal state}; moreover the couple
$\big(u^*(\cdot),X(\cdot;x,u^*(\cdot))\big)$ is called an {\em optimal couple}.

\section{Generalized Dynkin's formula}\label{SS:DYN}
The aim of the present section is to prove an abstract Dynkin formula for the controlled Ornstein-Uhlenbeck process \eqref{OU} composed with suitably smooth functions $\varphi:H\to\R$.

{\subsection{Transition semigroups, generators and $G$-derivatives}\label{SSS:DYN1}
}

We consider the family of transition semigroups associated to the uncontrolled version of \eqref{OU} and to the same process under  constant controls.  Precisely, we denote by $X^{(k)}(\cdot;x)$, where $k\in K$, the Ornstein-Uhlenbeck process starting at $x\in H$ with extra drift $Gk$; i.e., the mild solution to
\begin{equation}
\begin{cases}
dX(t)=\big[AX(t) + Gk\big]\,dt
+ \sigma dW(t), \ \ \ t\geq 0,\\
\label{eq4:SE1ellbisnew}
X(0)=x.
\end{cases}
\end{equation}
Its explicit expression is
 \begin{equation}\label{OUnew}
X^{(k)}(t;x):=e^{tA}x+ \int_0^t e^{(t-s)A}\sigma\, dW(s)+\int_0^t \overline{e^{(t-s)A} G}k\,ds.
\end{equation}
Correspondingly, we define the family of linear operators
$\big\{P^{(k)}_t\big\}_{t\geq 0}$  in the space $C_b(H)$ as
\begin{equation}\label{eq:semi}
P^{(k)}_t[\varphi](x) :=\E\big[\varphi(X^{(k)}(t;x))\big],  \ \ \ \ \varphi\in C_b(H),\ x \in H, \ t\geq 0.
\end{equation}
%We  recall (see e.g. \cite[Theorem 5.2]{DaPratoZabczyk14}) that, given $t\geq 0$, the stochastic convolution $\int_0^t e^{(t-s)A}\sigma dW(s)$
%is a centered Gaussian random variable with law $\caln_{Q_t}$, where
%$Q_t:=\int_0^t e^{sA}\sigma \sigma^* e^{sA^*}ds$ is the covariance operator. Hence we can write
%\begin{equation}\label{eq:semibis}
%P^{(k)}_t[\varphi](x)=
%\int_H \varphi\left(y+e^{tA}x+{\red{\int_0^t\overline{e^{(t-s)A}G} kds}}\right)\caln_{Q_t}(dy).
%\quad t\ge 0,\; x \in H,\; k\in K.
%z\end{equation}
{In Proposition \ref{prop:der}(i) below
%\cite[Corollary 9.15]{DaPratoZabczyk14})
we will show that the family $\big\{P^{(k)}_t\big\}_{t\geq 0}$ is  a one-parameter semigroup of linear operators in the space $C_b(H)$. {According to the related the literature, we call it} the \emph{transition semigroup} associated to  the process $X^{(k)}$.}
Unfortunately, such semigroup is not in general a $C_0$-semigroup in $C_b(H)$, {not even in the case
%of the uncontrolled Ornstein-Uhlenbeck process, i.e. in the case
$k=0$}.
%, even in the simplest case when $H=\R$ and $A\ne 0$. Indeed, see e.g. \cite{Cerrai94}, it is not true that, for any $\phi\in C_b(H)$, $\lim_{s\to 0^+}P^{(k)}_s\phi =\phi$ in $C_b(H)$ and this problem arises also in many other functional spaces \red{SPECIFICARE??}.
Indeed, in the framework of spaces of functions not vanishing at infinity, the $C_0$-property, i.e. the fact that $\lim_{s\to 0^+}P^{(k)}_s\varphi =\varphi$ {in the sup norm}
%\red{in functional sense}
for every $\varphi$, fails even in  basic cases.
For instance, this property fails in the case of the Ornstein-Uhlenbeck semigroup in the space $C_b(\mathbb{R})$ (see, e.g., \cite[Example 6.1]{Cerrai94} for a counterexample in $UC_b(\R)$, or  \cite[Lemma 3.2]{DaPratoLunardi95}, which implies this is a $C_0$-semigroup  in $UC_b(\R)$ if and only if the drift of the SDE vanishes). Even worse: given $\varphi\in C_b(H)$,  the map $[0,+\infty)\to C_b(H)$, $t\mapsto P_t^{(k)}\varphi$  is not in general measurable, as shown in \cite[Example 4.5]{FedericoGozzi16}.
This prevents, for instance, to intend in Bochner sense, in the space $C_b(H)$ for each $g\in C_b(H)$, {the integral defining} the Laplace transform
 \begin{equation}\label{eq:4resolventintegral1bis}
 \int_0^{\infty} e^{-\lambda s} P^{(k)}_s [g] ds.
\end{equation}
Nevertheless, one can get, in a weaker sense,
%pointwise sense, for each $x\in H$,
 several statements of the classical theory of $C_0$-semigroups.
This is performed, e.g., by the theory of $\calk$-semigroups
(introduced in \cite{Cerrai94}, see also \cite{CerraiGozzi95}, with the different terminology of \emph{weakly continuous semigroups})
%\cite{CerraiGozzi95}))
and  $\pi$-semigroups
(introduced in \cite{Priola99, Priola99-thesis}).
Both theories (a survey of which can be found in Appendix B.5 of \cite{FGSbook})
can be applied here getting substantially the same results. We employ the $\pi$-semigroups approach, as it
 seems more natural in our context.
%the definition of generator given in the context of $\pi$-semigroups is more suitable to prove our key result Proposition \ref{lemma:sem-spezz}.
%\begin{Remark}
%The theory of weakly continuous semigroups \cite{Cerrai94} (also called $\calk$-semigroups \cite{CerraiGozzi95})  is also employable here. In our context the approach through $\pi$-semigroups seems more natural with respect to the former one, as the definition of generator given in the context of $\pi$-semigroups is more suitable to prove our key result Proposition \ref{lemma:sem-spezz}.
%
%Also the more recent theory of $C_0$-sequentially equicontinuous semigroups in locally convex spaces \cite{FedericoRosestolato16} can be employed, but it goes beyond the needs of this paper, which only deals with the Ornstein-Uhlenbeck process. {Such theory might be useful
%in extending the results of this paper to more general types of controlled processes.}
%\end{Remark}
%\begin{Remark}
%The theory of weakly continuous semigroups \cite{Cerrai94} (also called $\calk$-semigroups \cite{CerraiGozzi95})  is also employable here. In our context the approach through $\pi$-semigroups seems more natural with respect to the former one, as the definition of generator given in the context of $\pi$-semigroups is more suitable to prove our key result Proposition \ref{lemma:sem-spezz}.
%\end{Remark}
The definition of $\pi$-convergence can be found e.g. in \cite[p.\,111]{EthierKurtz86},  where it is called $\textsl{bp}$-convergence (\emph{bounded-pointwise} convergence) and in \cite{Priola99, Priola99-thesis}; the former in the space $C_b(H)$, the latter in the space $UC_b(H)$.
%For the definition of
%$\mathcal{K}$-convergence in the space $UC_b$ the reader is referred to \cite{Cerrai94}, \cite{CerraiGozzi95}; for the space $C_m$ (respectively, $UC_m$) the reader is referred to \cite{GoldysKocan01}
%(respectively, to \cite{Cerrai95}\footnote{In that paper, the $\calk$-convergence in $UC_m$ is called $\calk_m$-convergence.}).
%%The same definitions hold if
%the $C_m(H)$ is replaced by  $UC_m(H)$.
%For instance for the definitions of $\pi$-convergence and $\mathcal{K}$-convergence
%in  $C_m(H)$ (for any $m>0$) we would simply substitute
%$C_b(H)$ with $C_m(H)$ (and so its norm $\|\cdot\|_0$ with $\|\cdot\|_{C_m(H)}$)
%and
%$C_b(I \times H)$ with $C_m(I \times H)$ (and so its norm $\|\cdot\|_0$ with $\|\cdot\|_{C_m(I \times H)}$).
%\footnote{Clearly these cases would include the definitions when the basic spaces are the subspaces $UC_b(H)$ and $UC_m(H)$.}.}

\begin{Definition}[$\pi$-convergence]\label{dfB:piconvergence}
A sequence of functions $(f_n)\subseteq C_b(H)$ is said to be $\pi$-convergent to a function $f\in C_b(H)$
if
$$  \sup_{n\in\mathbb{N}} |f_n|_{C_b(H)} < \infty \ \ \mbox{and} \ \ \lim_{n\to\infty} f_n(x) = f(x), \quad \forall x\in H.
$$
Such convergence is denoted by
$
f_n \xrightarrow{\pi} f$ or by $f = \pi\mbox{-}\!\!\lim_{n\to \infty} f_n.
$
\end{Definition}

Now we recall the definition of $\pi\mbox{-}$semigroup as given  in \cite{Priola99, Priola99-thesis}. Here we state it in the space of continuous and bounded functions (the aforementioned references deal with the space of \emph{uniformly} continuous and bounded functions, {but also explain how to extend the definition to $C_b(H)$}).
\begin{Definition}\label{2015-10-30:01bis}
A semigroup  $\big\{P_t\big\}_{t\geq 0}$ of bounded linear operators on $C_b(H)$
is called a $\pi\mbox{-}$semigroup on $C_b(H)$ if it satisfies the following conditions.
\begin{enumerate}[(P1)]
\item\label{2015-10-13:09bis} There exist $M\geq 1$ and $\alpha\in\R$ such that $|P_t[f]|_\infty\leq M e^{\alpha t}|f|_\infty$ for every $t\in \mathbb{R}^+$, $f\in C_b(H)$.
  \item\label{2015-10-30:04bis} For each $x\in H$ and $f\in C_b(H)$, the map $\mathbb{R}^+\rightarrow \mathbb{R}, \ t \mapsto P_t[f](x)$ is continuous.
  \item\label{2015-10-30:06bis} We have
   $$\{f_n\}_{n\in \mathbb{N}}\subset C_b(H), \ \ f_n \xrightarrow{\pi} f \in C_b(H) \ \Longrightarrow \ P_t [f_n] \xrightarrow{\pi} P_t[f]\, \ \ \forall t\geq 0.$$
  \end{enumerate}
\end{Definition}
Define
\begin{equation}\label{eq:priola1}
\cald(\cala^{(k)}):=\left\{\varphi\in C_b(H): \ \exists\  \pi\mbox{-}\!\!\lim_{t\rightarrow 0^+} \frac{P_t^{(k)}[\varphi]-\varphi}{t}\right\}
\end{equation}
and
\begin{equation}\label{eq:priola2}
\cala^{(k)}[\varphi] := \pi\mbox{-}\!\!\lim_{t\rightarrow 0^+} \frac{P_t^{(k)}[\varphi]-\varphi}{t},  \ \ \ \ \  \varphi \in \cald(\cala^{(k)}).
\end{equation}
{{It is proved  (see \cite[Lemma.\,5.7]{CerraiGozzi95} combined with the discussion of \cite[Sec.\,4.3]{Priola99}) that, for $\varphi$ sufficiently smooth,
\begin{equation}\label{aaa}
\cala^{(0)} [\varphi] (x) = \frac12 {\rm Tr}\left[\sigma \sigma^* D^2 \varphi (x)\right] +
\<x,A^*D\varphi (x)\>.
\end{equation}
We will use \eqref{aaa} to formally motivate the definition of mild solution  (Definition \ref{df4:solmildHJBell})
%The
of the HJB equation associated to the control problem of Section 3.}}

\begin{Proposition}\label{prop:der}
Let $k\in K$.
\begin{enumerate}[(i)]
\item
The family of linear operators  $\big\{P_t^{(k)}\big\}_{t\geq 0}$ defined in \eqref{eq:semi} is a $\pi$-semigroup on  $C_b(H)$.
We denote by $\cala^{(k)}$ its infinitesimal generator.
%is defined as
%\begin{equation}\label{eq:priola1}
%\cald(\cala^{(k)}):=\left\{\varphi\in C_b(H): \ \exists\  \calk\mbox{-}\!\!\lim_{t\rightarrow 0^+} \frac{P_t^{(k)}[\varphi]-\varphi}{t}\right\},
%\end{equation}
%\begin{equation}\label{eq:priola2}
%\cala^{(k)}[\varphi] := \calk\mbox{-}\!\!\lim_{t\rightarrow 0^+} \frac{P_t^{(k)}[\varphi]-\varphi}{t},  \ \ \ \ \  \varphi \in \cald(\cala^{(k)}).
%\end{equation}
\item
The  operator
%$
%R^{(k)}_\lambda: C_b(H)\to C_b(H)$
%defined by
$$
R^{(k)}_\lambda[g](x):=\int_0^{\infty} e^{-\lambda s} P^{(k)}_s [g](x) ds, \ \ \ \  g\in C_b(H),  \  x\in H,
$$
belongs to $\call (C_b(H))$ for every $\lambda>0$ and is the resolvent of $\cala^{(k)}$:
$$\big(\lambda-\cala^{(k)}\big)^{-1}= R_\lambda^{(k)}, \ \ \ \forall \lambda>0.$$
\item
We have\footnote{At $t=0$ the derivative is intended as right derivative.}
$$\frac{d}{dt}P^{(k)}_t[\varphi](x)=P_t^{(k)} \big[\cala^{(k)} [\varphi]\big](x)=\cala^{(k)}\big[P_t^{(k)} [ \varphi]\big](x),  \ \ \
\forall \varphi\in \cald(\cala^{(k)}), \ \forall x\in H, \ \forall t\geq 0.$$
\end{enumerate}
\end{Proposition}
\begin{proof}
Claims (ii)-(iii) follow from
 \cite[Prop.\,3.2,\ Prop.\,3.6]{Priola99} or \cite[Prop.\,6.2.7, Prop.\,6.2.11]{Priola99-thesis}(\footnote{These references deal mainly in the space of uniformly continuous and bounded functions --- we warn that the author denotes by $C_b(H)$ the latter space. The extension to  the space of continuous and bounded function --- our space $C_b(H)$ --- is illustrated in \cite[Sec.\,5]{Priola99} and \cite[Sec.\,6.5]{Priola99-thesis}.}) once one proves claim (i), which we prove below.

\emph{Proof of (i).}  {{First of all, we prove that $\big\{P_t^{(k)}\big\}_{t\geq 0}$ is a semigroup of linear operators on $C_b(H)$. The fact that $P_0^{(k)}=I$ and that $P_t^{(k)}\in \call(C_b(H))$ for all $t\geq 0$ is immediate. The semigroup property of $\big\{e^{tA}\big\}_{t\geq 0}$ and \eqref{eq:semi1}
yield
$$
X^{(k)} (t+s;x)= e^{sA} X^{(k)}(t;x)+\int_{0}^s   e^{(s-r)A}\sigma\, dW(t+r) + \int_{0}^s   \overline{e^{(s-r)A}G} k\,dr, \ \ \ \forall t\geq 0, \forall s >0.
$$
The latter shows the strong Markov property of $X^{(k)}$ and then the fact that $\big\{P^{(k)}_t\big\}_{t\geq 0}$ satisfies the semigroup property follows as consequence (see, e.g., \cite[Cor.\,9.15]{DaPratoZabczyk14}).}}
\smallskip
Now we show  the other properties of Definition \ref{2015-10-30:01bis}.
(P1) is obviously verified with $M=1$ and $\alpha=0$.
%The proof of this fact is in \cite[Subsec.\,6.3.3]{Priola99-thesis} (see also \cite[Subsec.\,4.3]{Priola99}).
%Here we show the extension to $k\neq 0$.
 (P2) of Definition \ref{2015-10-30:01bis} corresponds to
\begin{equation}\label{OUcont}
\E\big[f(X^{(k)}(t;x)\big]\
 \stackrel{t\to t_0}{\longrightarrow}\ \E\big[f(X^{(k)}(t_0;x)\big], \ \ \ \forall f\in C_b(H), \ \forall x\in H, \ {\forall t_0\ge 0}.
\end{equation}
The latter follows from continuity of trajectories of  $X^{(k)}(\cdot;x)$ and  dominated convergence.
Finally, (P3)  of Definition \ref{2015-10-30:01bis} is verified by dominated convergence.
\end{proof}

\color{black}

A key step  towards the main goal of this section, i.e. the proof of a  generalized Dynkin formula for $\varphi(X(\cdot;x,u(\cdot))$ with a suitably regular $\varphi$, consists in showing the following decomposition of $\mathcal{A}^{(k)}$ when acting on the function $\varphi$
\begin{equation}\label{eq4:Aksplit}
\varphi\in \mathcal{D}(\mathcal{A}^{(0)})\cap C^{1,G}_b(H) \ \Longrightarrow \ \varphi\in \mathcal{D}(\mathcal{A}^{(k)}) \ \mbox{and}   \
\cala^{(k)}[\varphi ]=
\cala^{(0)}[\varphi]+\langle D^G\varphi(\cdot),k\rangle_K \ \ \ \ \ \forall k\in K.
\end{equation}
Looking at $\big\{P_t^{(k)}\big\}_{t\geq 0}$ as to a perturbation  of $\big\{P_t^{(0)}\big\}_{t\geq 0}$, \eqref{eq4:Aksplit} is obtained  in
{{\cite[Theorem 5.2]{GoldysKocan01}
in the context of $C_0$-semigroups with respect to \emph{mixed topology} of $C_b(H)$ and in \cite[Theorem 4.6]{farkas} in the context of \emph{bi-continuous} semigroups.
However, these references would require the assumptions that  $\varphi \in C^1_b(H)$ and $A,\sigma$ are such that  $C^1_b(H)\subseteq \mathcal{D}(\mathcal{A}^{(0)})$ and 
$G\in \call(H)$. This would allow, in particular, to write the term $\langle D^G\varphi(\cdot),k\rangle_K$ in the formula above    as $\langle D\varphi(\cdot),Gk\rangle_H$, simplifying  a lot the framework.}}
%---  to be more precise, under assumptions on $A,\sigma$ ensuring that $C^1(H)\subseteq \mathcal{A}^{(0)}$ --- and, in terms of our framework, when $G\in \call(H)$ --- so that the term $\langle D^G\varphi(\cdot),k\rangle_K$ in the formula above can be read as $\langle D\varphi(\cdot),Gk\rangle_H$.
Here we need to be sharper in this respect in order to cover other cases of interest in applications, e.g., the case of unbounded $G$, occurring in boundary control problems. To this purpose we introduce
 the class of functions
{\small{\begin{equation}\label{class:fun}
\begin{split}
\mathcal{S}^{A,G}(H):=\left\{\varphi\in C^{1,G}_b(H): \
\lim_{t\to 0^+} \frac{\varphi\left( z(t)+\dis\int_0^t\overline{e^{sA}G}k ds\right)- \varphi(z(t))}{t}=\left\langle D^G \varphi(z(0)),k\right\rangle_K \ \forall z\in C(\R^+;H) \right\}.
\end{split}
\end{equation}}}
Our generalized Dynkin formula will hold for functions belonging to $\mathcal{D}(\cala^{(0)})\cap \mathcal{S}^{A,G}(H)$. In Appendix \ref{sec:app} we  provide sufficient  conditions on $A,G, \varphi$  ensuring that $\varphi\in \mathcal{S}^{A,G}(H)$.

\begin{Proposition}\label{lemma:sem-spezz} Let $\varphi\in \mathcal{D}(\cala^{(0)})\cap \mathcal{S}^{A,G}(H)$.  Then
\eqref{eq4:Aksplit} holds.
% $\varphi\in \mathcal{D}(\mathcal{A}^{(0)})\cap \mathcal{S}_1$. Then $\varphi\in \cald(\cala^{(k)})$ for every $k\in K$ and
%\begin{equation}\label{eq4:Aksplit}
%\cala^{(k)}[\varphi ]=\cala^{(0)}[\varphi]+\langle D^G\varphi(\cdot),k\rangle_K.
%\end{equation}
\end{Proposition}
\begin{proof}
%Now, assume first that $k\in \cald(G)$. Then
% \begin{equation}\label{stima1}
% \lim_{t\to 0^+} \frac{1}{t}\left(\int_0^t(e^{(t-s)A}Gk-Gk)ds	\right)=0.
% \end{equation}
Since $\varphi\in \cald(\cala^{(0)})$,  we can write for every $x\in H$
\begin{align*}
{\cala^{(k)}[\varphi](x)=}\lim_{t\to 0^+} \frac{P_t^{(k)}[\varphi](x)-\varphi(x)}{t}&= \lim_{t\to 0^+} \frac{P_t^{(k)}[\varphi](x)-P_t^{(0)}[\varphi](x)}{t}+\lim_{t\to 0^+} \frac{P_t^{(0)}[\varphi](x)-\varphi(x)}{t}\\
&= \lim_{t\to 0^+} \frac{\E\left[\varphi(X^{(k)}(t;x))-\varphi(X^{(0)}(t,x))\right]}{t}
+\cala^{(0)}[\varphi](x),
\end{align*}
if the last limit exists.
Observe that
 $$X^{(k)}(t;x)=X^{(0)}(t;x)+\int_0^t\overline{e^{(t-s)A}G}k\,ds= X^{(0)}(t;x)+\int_0^t\overline{e^{sA}G}k\,ds.$$ Therefore, since $\varphi\in\mathcal{S}^{A,G}(H)$,  continuity  of $t\mapsto \int_0^t\overline{e^{sA}G}k\, ds$ and by dominated convergence yield
\begin{align*}
& \lim_{t\to 0^+} \frac{\E\left[\varphi(X^{(k)}(t;x))-\varphi(X^{(0)}(t;x))\right]}{t}= \lim_{t\to 0^+}\E\left[\frac{\varphi\left(X^{(0)}(t;x)+
\dis\int_0^t\overline{e^{sA}G}k ds \right)-\varphi(X^{(0)}(t;x))}{t}\right]\\&
\E\left[\lim_{t\to 0^+}\frac{\varphi\left(X^{(0)}(t;x)+
\dis\int_0^t\overline{e^{sA}G}k ds\right)-\varphi(X^{(0)}(t;x))}{t}\right]
%&=\E\left[\lim_{t\to 0^+}\frac{\varphi\left(X^o(t;x)+tGL(u) \right)-\varphi(X^o(t,x))}{t}\right]\\
%&=\E\left[\lim_{t\to 0^+}\frac{\varphi\left(X^o(t;x)+tGL(u) \right)-\varphi(X^o(t,x)+tGL(u))}{t}\right]\\
%+\E\left[\lim_{t\to 0^+}\frac{\varphi\left(X^o(t;x)+tGL(u) %\right)-\varphi(X^o(t,x))}{t}\right]
= \big\langle D^G\varphi(x),k\big\rangle_K.
\end{align*}
The claim follows.
\end{proof}

{
\subsection{Proof of the generalized Dynkin's formula}\label{SSS:DYN2}
}
We introduce the linear space  $\calk^{s,p}$ of $K$-valued $p$-integrable c\`adl\`ag simple  processes. An element $\kappa(\cdot)\in \calk^{s,p}$ is of the form
\begin{equation}\label{kkkbis}
\kappa(t)=\sum_{i=1}^n {k_{i-1}}\textbf{1}_{[t_{i-1},t_i)}(t),
\end{equation}
for some $n\in \N$,  $0={t_0}<t_1<...<t_n=+\infty,$ and  $\{k_i\}_{i=0,...,{n-1}}$ such that  $k_i\in L^p(\Omega,\calf_{t_i},\P;K)$  for all $ i=0,...,{n-1}$. Processes in $\calk^{s,p}$ are progressively measurable.
{By arguing as in the proof of
%Lemma \ref{lemmaDP},
Proposition \ref{prop:SE} we get that,} for any
$\kappa(\cdot)\in \calk^{s,p}$, the process
$$
t\mapsto \int_0^t \overline{e^{(t-s)A} G}\kappa(s)\,ds
$$
is well defined, belongs to $\calm^{1,loc}_\calp(H)$ and has a version with continuous trajectories.
%for each $\kappa(\cdot)\in \calk^{s,p}$.
We will always refer to the version of this process (unique up to indistinguishability) having continuous trajectories.
Given $\kappa(\cdot)\in\calk^{s,p}$, we write
$$
X^{\kappa(\cdot)}(t,x):=e^{tA}x+ \int_0^t e^{(t-s)A}\sigma \,dW(s)+\int_0^t \overline{e^{(t-s)A} G}\kappa(s)\,ds.
$$
{Again arguing as in the proof of Proposition \ref{prop:SE} we see that this process has a version with having continuous trajectories. As above we will always refer to this version (unique up to indistinguishability).}

Recall that, if $V_1$, $V_2$ are two random variables with values, respectively, in two measurable spaces $(E_1,\cale_1)$ and $(E_2,\cale_2)$, a version of the conditional law of $V_1$ given $V_2$ is a family of probability measures $\big\{\mu(\cdot,v_2)\big\}_{v_2\in E_2}$ on $(E_1,\cale_1)$ such that, for every $f\in B_b(E_1\times E_2;\R)$, the map $v_2\mapsto \int_{E_1} f(v_1,v_2)\mu(dv_1,v_2)$ is measurable and
$$
\E[f({V_1,V_2})]= \int_{E_2}\nu(dv_2)\int_{E_1} f(v_1,v_2)\,\mu(dv_1,v_2),
$$
where $\nu=\textsl{Law}\,(V_2)$. This family, if it exists, is unique up to $\nu$-null measure sets.
\smallskip
\begin{Lemma}\label{lemma:imp}
Let $\kappa(\cdot)\in \calk^{s,p}$ be in the form \eqref{kkkbis} and $t\in [t_{i-1},t_{i})$ for some $i=1,...,n.$
A  version of the conditional law of $X^{\kappa(\cdot)}(t;x)$ given the couple $(X^{\kappa(\cdot)}(t_{i-1};x), k_{i-1})$ is the family
\begin{equation}\label{eqmutdef}
\mu_t(\cdot,x',k')\,{:}=\textsl{Law}\  (X^{(k')}(t-t_{i-1};x')).
\end{equation}
\end{Lemma}
\begin{proof}
The proof is standard (see \cite[Ch.\,2,\,Sec.\,9]{Krylov80} in finite dimension and in a much more general setting) and we omit it for brevity.
\end{proof}

\begin{Lemma}\label{lemma:sem-spezz2}
Let $\varphi\in\mathcal{D}(\cala^{(0)})\cap \mathcal{S}^{A,G}(H)$  and  $\kappa(\cdot)\in \calk^{s,p}$.
Then
\begin{equation}\label{ddx}
\frac{d}{dt}\,\E\left[\varphi\big(X^{\kappa(\cdot)}(t;x)\big)\right]=
\E\left[\cala^{(0)}[\varphi]\big(X^{\kappa(\cdot)}(t;x)\big)+\left\langle \kappa(t),\,D^G\varphi\big(X^{\kappa(\cdot)}(t;x)\big)\right\rangle_K\right], \ \ \ \forall t\geq 0,
\end{equation}
where the derivative has to be intended as right derivative at the times $\{t_1,...,t_n\}$,  where the simple process {$\kappa(\cdot)$} jumps.
\end{Lemma}
\begin{proof}
%Let
%$$k(t)=\sum_{i=1}^n k_{i-1}\textbf{1}_{[t_{i-1},t_i)}(t), \ \ \ t\geq 0,$$
%for some $n\in \N$,  $0={t_0}<t_1<...<t_n,$ and  $\{k_i\}_{i=0,...,n}$ such that $k_i\in L^p(\Omega,\calf_{t_i},\P;K)$  for all $ i=0,...,n$.
Let $\kappa(\cdot)\in \calk^{s,p}$ be as in  \eqref{kkkbis},  $t\in [t_{i-1},t_{i})$ for some $i=1,...,n$, and  
 $\varphi\in \cald(\cala^{(0)})\cap C^{1,G}_b(H)$.
%and let $K(\cdot)\in \calk^s$.
Denote by $\nu$ the law of the couple  $(X^{\kappa(\cdot)}(t_{i-1};x), k_{i-1})$.
By Lemma \ref{lemma:imp}, we have
%\footnote{The derivative is intended as right-derivative at $t=t_{i-1}$. The condition allowing differentiation under integral sign is verified a posteriori, as the integrand in the right hand side is seen to be bounded at the end.}
\begin{align*}
\E\left[\varphi\big(X^{\kappa(\cdot)}(t;x)\big)\right]
&=\int_{H\times K} \nu(dx',dk')\int_H\varphi(\xi)\mu_t(d\xi;x',k')
\\
&= \int_{H\times K} \nu(dx',dk')\;\E\left[\varphi \big(X^{(k')}(t-t_{i-1};x')\big)\right]
\\
&=\int_{H\times K}  \nu(dx',dk')P_{t-t_{i-1}}^{(k')}[\varphi](x')
\end{align*}
Now we differentiate under the integral sign using the fact that, by Proposition \ref{lemma:sem-spezz}, $\varphi \in \cald(\cala^{(k')})$ and the fact that
$(t,x')\mapsto P_{t-t_{i-1}}^{(k')}[\cala^{(k')}[\varphi]](x')$ is bounded over $[t_{i-1},t_i)\times H$.
Then, using Proposition \ref{prop:der}(i) and \myref{eq4:Aksplit}, we get
\begin{align*}
\frac{d}{dt}\E\left[\varphi\big(X^{\kappa(\cdot)}(t;x)\big)\right]
&=\int_{H\times K}  \nu(dx',dk')\,\frac{d}{dt} P_{t-t_{i-1}}^{(k')}[\varphi](x')\\
&=\int_{H\times K}  \nu(dx',dk')\,\, P_{t-t_{i-1}}^{(k')}[\cala^{(k')}[\varphi]](x')
\\
&=\int_{H\times K}  \nu(dx',dk') \;\; \E\left[\cala^{(k')}[\varphi](X^{(k')}(t-t_{i-1});x)\right]
\\
&= \int_{H\times K}  \nu(dx',dk')\,\,
\int_H\cala^{(k')}[\varphi](\xi)\mu_t(d\xi;x',k')
\\
&= \E\left[\cala^{(k_{i-1})}[\varphi]\big(X^{\kappa(\cdot)}(t;{x})\big)\right]
\\
&=\E\left[\cala^{(0)}[\varphi]\big(X^{\kappa(\cdot)}(t;x)\big)+\left\langle k(t),\,D^G\varphi(X^{\kappa(\cdot)}(t;x)\big)\right\rangle_K\right],
\end{align*}
%\begin{align*}
%\frac{d}{dt}\E\left[\varphi(X^{K(\cdot)}(t;x)\right]
%&= \frac{d}{dt}\int_{H\times K} \nu(dx',dk)\int_H\varphi(\xi)\mu_t(d\xi;x',k)\\
%%&{=}\E\left[\frac{d}{dt}\varphi(X(t;x,u(\cdot))\right]\\
%&= \int_{H\times K} \nu(dx',dk)\frac{d}{dt}\int_H\varphi(\xi)\mu_t(d\xi;x',k)\\
%&=\int_{H\times K}  \nu(dx',dk) \frac{d}{dt}P_t^{(k)}[\varphi](x')\\
%&=\int_{H\times K}  \nu(dx',dk) P_t[\cala^{(k)}[\varphi]](x')\\
%&=\int_{H\times U}  \nu(dx',dk) \E\left[\cala^{(k)}[\varphi](X^{(k)}(t-t_i);x)\right]\\
%&= \E\left[\cala^{(K_{i-1})}[\varphi](X^{K(\cdot)}(t;x))\right]
%\\
%&=\E\left[\cala^{(0)}[\varphi](X^{K(\cdot)}(t;x)+\langle K(t),D^G\varphi(X^{K(\cdot)}(t;x)\rangle_K\right],
%\end{align*}
the claim.
\end{proof}

\color{black}

\begin{Lemma}\label{lemma:conv}
For each $u(\cdot)\in \calu$ and $T>0$, there exists a sequence $\{\kappa_n\}_{n\in\N}\subset \calk^{s,p}$ such that
$$
\kappa_n|_{[0,T]\times\Omega}\ \stackrel{\calm^{p,T}_\calp(H)}{\longrightarrow}\  L(u(\cdot))|_{[0,T]\times\Omega},
\ \ \
X^{\kappa_n(\cdot)}(\cdot;x)|_{[0,T]\times\Omega}\ \stackrel{\calm^{1,T}_\calp(H)}{\longrightarrow} \ X(\cdot;x,u(\cdot))|_{[0,T]\times\Omega}.
$$
\end{Lemma}
\begin{proof} Fix $T>0$ and set $\kappa(\cdot):=L(u(\cdot))$.
By standard arguments (see, e.g., \cite[Ch.\,III,\,Lemma.\,2.4,\,p.132]{KaratzasShreve88})\footnote{It is worth to point out some differences. First, we are dealing with c\`adl\`ag approximations (as it is more meaningful and natural to state Proposition \ref{lemma:sem-spezz2}) rather than with c\`agl\`ad (as in \cite[Ch.\,III,\,Lemma.\,2.4,\,p.132]{KaratzasShreve88}): this is not a problem as, from the point of view of integration, these classes coincide. Second, we are dealing with Hilbert-valued processes: therefore, more technical care is needed as the approximation is produced by Bochner integration.}, we can construct a sequence $\{\kappa_n\}_{n\in\N}\subset \calk^{s,p}$ such that
$$\kappa_n|_{[0,T]\times\Omega}\ \stackrel{\calm^{p,T}_\calp(H)}{\longrightarrow}  \ \kappa(\cdot)|_{[0,T]\times\Omega}.$$
%$$
%\int_0^T \E[|k_n(s)-k(s)|^p_K]ds \to 0.
%$$
{Then, using the expression \eqref{OU} for the state variable,
%using Lemma \ref{lemmaDP},
the convergence
$$ X^{\kappa_n(\cdot)}(\cdot;x)|_{[0,T]\times\Omega}\ \stackrel{\calm^{1,T}_\calp(H)}{\longrightarrow}  \ X(\cdot;x,u(\cdot))|_{[0,T]\times\Omega}$$ follows by simply applying dominated convergence.}
%$$
%\int_0^t\overline{e^{sA}G}k_n(s)ds \stackrel{L^p(\P)}{\longrightarrow} \int_0^t\overline{e^{sA}G}k(s)ds= \int_0^t\overline{e^{sA}G}L(u(s))ds, \ \ \forall t \in [0,T],
%$$
%and the claim follows.
\end{proof}

%In the following, by $\frac{d^+}{dt}$ we will denote the right derivative of a real function defined on $[0,+\infty)$.

\begin{Theorem}[Dynkin's formula]\label{teo:dynkin}
Let $\varphi\in\mathcal{D}(\cala^{(0)})\cap \mathcal{S}^{A,G}(H)$. Then, for every $\lambda>0$, $T>0$, and $u(\cdot)\in \calu$, we have
\begin{align}\label{Dynkin}
&\E\left[e^{-\lambda T}\varphi\big(X(T;x,u(\cdot))\big)\right]\\&=
\varphi(x)+
\E\left[\int_0^T e^{-\lambda t}
\left[(\cala^{(0)}-\lambda)[\varphi ]\big(X(t;x,u(\cdot))\big)
+\< L(u(t)),\,D^G\varphi(X(t;x,u(\cdot)))\>_K\right]dt\right].\nonumber
\end{align}
\end{Theorem}
\begin{proof}
Let $u(\cdot)\in \calu$ and
take the approximating sequence $\left\{\kappa_n\right\}_{n\in\N}$ provided by Lemma \ref{lemma:conv}. Then, applying, for each $n\in \N$, Lemma \ref{lemma:sem-spezz2},
we obtain from \myref{ddx} (by taking the right derivatives at  $t_i$), for all $t \ge 0$ and $\lambda >0$,
%\begin{equation}\label{ddxn}
%\frac{d}{dt}\E\left[\varphi(X^{K_n(\cdot)}(t;x)\right]=
%\E\left[\cala^{(0)}[\varphi](X^{K_n(\cdot)}(t;x)+\langle K_n(t),D^G\varphi(X^{K_n(\cdot)}(t;x)\rangle_K\right], \ \ \ \forall t\geq 0,
%\end{equation}
%Then, by straightforward computations, we get, for every $\lambda >0$,
\begin{align}\label{ddxnlambda}
&\frac{d}{dt}e^{-\lambda t}\E\left[\varphi\big(X^{\kappa_n(\cdot)}(t;x)\big)\right]\\&=
-\lambda e^{-\lambda t}\E\left[\varphi\big(X^{\kappa_n(\cdot)}(t;x)\big)\right]+
e^{-\lambda t}\E\left[\cala^{(0)}[\varphi]\big(X^{\kappa_n(\cdot)}(t;x)\big)+\left\langle \kappa_n(t),\,D^G\varphi\big(X^{\kappa_n(\cdot)}(t;x)\big)\right\rangle_K\right].\nonumber
%\ \ \ \forall t\geq 0,
\end{align}
{Since the function $t\mapsto \E\left[e^{-\lambda t}\varphi\big(X^{\kappa_n(\cdot)}(t;x)\big)\right]$ is everywhere continuous and stepwise differentiable, we can apply the Fundamental Theorem of Calculus.} So,
integrating on $[0,T]$, we get
\begin{align*}
&\E\left[e^{-\lambda T}\varphi\big(X^{\kappa_n(\cdot)}(T;x)\big)\right]\\&=\varphi(x)+\E\left[\int_0^T e^{-\lambda t}\bigg((\cala^{(0)}-\lambda)[\varphi ]\big(X^{\kappa_n(\cdot)}(t;x)\big)+\left\langle \kappa_n(t), \,D^G\varphi\big(X^{\kappa_n(\cdot)}(t;x)\big)\right\rangle_K\bigg)\,dt\right].
\end{align*}
Now, letting $n\to + \infty$, we get the claim by dominated convergence from Lemma \ref{lemma:conv}, observing that
 $\varphi$, $D^G\varphi$, and $\cala^{(0)}[\varphi]$  are bounded.
%It follows by applying Proposition \ref{lemma:sem-spezz2} to the approximating sequence $K_n$ provided by Lemma \ref{lemma:conv} and then passing the equality to the limit by dominated convergence (all terms are indeed bounded).
%%For processes  $u(\cdot)\in \calu^s$ the formula is aded  consequence of Lemma \ref{lemma:sem-spezz2}. For general processes $u(\cdot)\in\calu$, one approximate them by employing Lemma \ref{lemma:conv} and passes to the limit the already proved formula for simple processes.
\end{proof}

\begin{Remark}
\label{rm:polgrowth}
The results of this section, in particular Theorem \ref{teo:dynkin},
can be extended, at the price of straightforward technical complications, to the case when the basic space of functions is, instead of
$C_b(H)$, the space $C_m(H)$, where $m>0$, used e.g. in \cite{FedericoGozzi16}:
\begin{equation}\label{eq:defCm}
C_m(H):=\left\{\phi:H\to \R \ \mbox{continuous}: \  \sup_{x \in H}\frac{|\phi(x)|}{1+|x|^m}<\infty \right\}.
\end{equation}
Also the results of next Section 5 can be extended to this setting
covering more general cases, in particular when the current cost of the control problem has polynomial growth in $x$. We do not do this here for brevity.
\end{Remark}

\color{black}

\section{HJB equation, verification theorem and optimal feedbacks}\label{SS:HJB}
By standard Dynamic Programming arguments, one formally associates to
{the control problem of Section 3 the following HJB equation for the value function \eqref{eq4:VF1ell}:}
\begin{equation}
\label{eq4:HJBcontrolell}
%\left\{ \begin{array}{l}
\displaystyle{ \lambda v(x) -\frac{1}{2}\;
\mbox{\rm Tr}\;[QD^2v(x)] -\< Ax,Dv(x)\>_H- F(x,Dv(x))=0,
\quad x \in H,}
%\\[1mm]
%\displaystyle{\qquad \qquad \qquad \quad
%\; t \in [0,T), \; x \in H}
%\\[1mm]
%v(T,x)=g (x), \; x \in H,
%\end{array} \right.
\end{equation}
where $Q=\sigma\sigma^*$ and the Hamiltonian $F$ is defined by
\begin{equation}
\label{eq4:Hamiltoniancontrolell}
	F(x,p) \coloneqq \inf_{u \in \Lambda}\,\, F_{CV} (x,p;u), \ \ \ \ \  \ x\in H, \ p\in  H,
\end{equation}
where
\begin{equation}\label{FCV}
F_{CV} (x,p;u)
	\coloneqq
 \big\langle GL(u),p \big\rangle_H +l(x,u), \ \ \ \ \ \ x\in H, \ u\in \Lambda, \ p\in H.
\end{equation}
{Note that this definition is only formal as $GL(u)$ {may be not defined, {since} $L(u)$ may not belong to $\mathcal{D}(G)$}.
It is then} convenient 
to introduce the modified Hamiltonian
\begin{equation}
\label{eq4:F0controlell}
	F_0(x,q) \coloneqq \inf_{u \in \Lambda}\,\,F_{0,CV} (x,q;u), \ \ \ \ \ \  x\in H, \ q\in K,
\end{equation}
where
\begin{equation}\label{F0CV}
F_{0,CV} (x,q;u)\coloneqq
%\inf_{u \in U}\left\{
  \big\langle L(u),q \big\rangle_K +l(x,u),  \ \ \ \ \ \ x\in H, \ u\in \Lambda, \ q\in K.
\end{equation}
Observing that
$$
F(x,p)=F_0(x,G^*p), \ \ \ \forall p\in \cald(G^*),
$$
 \eqref{eq4:HJBcontrolell} can be formally rewritten as
\begin{equation}
\label{eq4:HJBcontrolellbis}
%\left\{ \begin{array}{l}
\displaystyle{ \lambda v(x) -\frac{1}{2}\;
\mbox{\rm Tr}\;[QD^2v(x)] -\< Ax,Dv(x)\>_H- F_0(x,D^Gv(x))=0,
\quad x \in H}.
%\\[1mm]
%\displaystyle{\qquad \qquad \qquad \quad
%\; t \in [0,T), \; x \in H}
%\\[1mm]
%v(T,x)=g (x), \; x \in H,
%\end{array} \right.
\end{equation}
Note that, in principle, $F_0$ may take the value $-\infty$ somewhere.
The concept of mild solution to \eqref{eq4:HJBcontrolell} relies on Proposition \ref{prop:der}(ii) and on \eqref{aaa}, inspiring an integral form of   \eqref{eq4:HJBcontrolellbis} through the use of the semigroup $\big\{P_s^{(0)}\big\}_{s\geq 0}$. 
\begin{Definition}
\label{df4:solmildHJBell} We say that a
function $v:H\to\R$ is a mild solution to \eqref{eq4:HJBcontrolellbis}
if it belongs to $C _{b}^{1,G}\left(H\right)$, {$F_0(\cdot,D^{G}v\left(\cdot\right))$
is bounded  and} $v$ solves the integral equation
\begin{equation}
v(x)= \int_0^{\infty }e^{-\lambda s}
P_s^{(0)} \left[  F_0(\cdot,D^{G}v\left(\cdot\right))\right](x)\,ds, \ \ \ \forall x\in H.
\label{eq4:solmildHJBell}
\end{equation}
\end{Definition}
\begin{Remark}\label{rem:exist}
The problem of existence and uniqueness of mild solutions for equations in the form \eqref{eq4:HJBcontrolellbis} is addressed in \cite{FedericoGozzi16} and in  \cite[Ch.\,4]{FGSbook}. In particular, existence and uniqueness of mild solutions is stated  for sufficiently large $\lambda>0$, under the following assumptions (see \cite[Cor.\,4.12, Th.\,3.8(ii)]{FedericoGozzi16} {with $m=0$}):
\begin{itemize}
\item[(A1)] $\overline{e^{tA}G}(K)\subseteq Q_{t}^{1/2}(H)$ for every $t > 0$, where $Q_t:=\dis\int_0^t  e^{sA}\sigma\sigma^{\ast}e^{sA^{\ast}}ds$.
\item[(A2)]  The operators\footnote{Here $Q_{t}^{-1/2}$ is the pseudo-inverse of $Q_{t}^{1/2}$.}
\begin{equation*}\label{eq4:defGammaG}
\Gamma_G(t):K\to H, \  \  \ \Gamma_G(t)\coloneqq Q_{t}^{-1/2}\overline{e^{tA}G}, \ \ t\geq 0,
\end{equation*}
which are well defined by (A1) and bounded by the closed graph theorem, are such that the map $t\mapsto |\Gamma_G(t)|_{\call(K,H)}$ belongs to $L^1_{loc}([0,+\infty),\R)$ and is bounded in a neighborhood of $+\infty$.
\item[(A3)] The Hamiltonian $F_0$ satisfies, for suitable $C_{F_0}>0$.
$$
\big|F_0(x,q_1)-F_0(x,q_2)\big| \le C_{F_0}|q_1-q_2|_K, \qquad \forall x\in H, \ \ \forall q_1,q_2 \in K,
$$
$$
\big|F_0(x,q)\big| \le C_{F_0}\big(1+|q|_K\big), \qquad \forall x\in H,\ \ \forall q \in K.
$$
\end{itemize}
{Some results in the case of locally Lipschitz Hamiltonian are available, up to now, only in special cases (see \cite[Sec.\,13.3.1]{DaPratoZabczyk14} and \cite{Cerrai01-40}).}
\end{Remark}

Due to Proposition \ref{prop:der}(ii), a mild solution $v$ of \eqref{eq4:HJBcontrolell} enjoys the property of being a solution to the same equation also in a differential abstract way, i.e., we have the following.
\begin{Proposition}\label{prop:strict} Let $v$ be a mild solution to (\ref{eq4:HJBcontrolellbis}). Then $v\in \cald(\cala^{(0)})$ and
\begin{equation}\label{strict}
\big(\lambda-\cala^{(0)}\big)[v](x) = F_0	\big(x,D^Gv(x)	\big), \ \ \ \forall x\in H.
\end{equation}
\end{Proposition}
{ \begin{proof}
Using Proposition \ref{prop:der}(ii), we rewrite
\myref{eq4:solmildHJBell} as
\begin{equation}
v(x)= \big(\lambda-\cala^{(0)}\big)^{-1}\left[  F_0\big(\cdot,D^{G}v\left(\cdot\right)\big)\right](x), \ \ \ \forall x\in H.
\label{eq4:solmildHJBellbis}
\end{equation}
This entails $v\in \cald(\cala^{(0)})$ and,
applying $\lambda-\cala^{(0)}$ to both sides, we see that $v$ solves \myref{strict}.
\end{proof}}
\begin{Remark}\label{rem:reg} By Proposition \ref{prop:strict} a mild solution $v$ to \eqref{eq4:HJBcontrolellbis}
belongs to $\mathcal{D}(\cala^{(0)})$. Hence, in order to apply Theorem \ref{teo:dynkin} to it, we only need to assume that $v\in\mathcal{S}^{A,G}(H)$. This is what we indeed assume in all the  next results of this section.\end{Remark}
\subsection{Verification theorem}\label{SUB:VER}
The proof of the verification theorem relies in the so called \emph{fundamental identity}.
 \begin{Proposition}[Fundamental identity]\label{prop:FI}
Let \myref{eq4:hpl} hold. Let $v$ be a mild solution to (\ref{eq4:HJBcontrolellbis}) and assume that $v\in \mathcal{S}^{A,G}(H)$.
Let $x\in H$ and let $u(\cdot)\in \calu$ be such that
\begin{equation}\label{eq:finite}
J(x;u(\cdot)):=\E\left[\int_0^{\infty} e^{-\lambda t} l\big(X(t;x,u(\cdot)),u(t)\big)\,dt\right]<\infty.
\end{equation}
Then
\begin{align}\label{FI}
v(x)&=J(x;u(\cdot))\\&+\E\left[\int_0^{\infty}e^{-\lambda t} \left(F_0\big(X(t;x,u(\cdot)),\, D^Gv\big(X(t;x,u(\cdot))\big)\big)- F_{0,CV}\big(X(t;x,u(\cdot)),\, D^Gv\big(X(t;x,u(\cdot))\big); u(t)\big)\right)dt \right].\nonumber
\end{align}
 \end{Proposition}
 \begin{proof}
%Observe first that, by Definition \ref{df4:solmildHJBell} and by Proposition \ref{prop:strict}, we have
%\begin{equation}
%\endv\in D(\cala^{(0)})\cap C_b^{1,G}$.
Let $x\in H$, $T>0$, and let $u(\cdot)\in \calu$ be such that \eqref{eq:finite} holds.
Using Proposition \ref{prop:strict} and applying the abstract Dynkin formula (Theorem \ref{teo:dynkin}) to
$t\mapsto e^{-\lambda t}v(X(t;x,u(\cdot)))$, we get
\begin{align}\label{Dynkinv}
&\E\left[e^{-\lambda T}v\big(X(T;x,u(\cdot))\big)\right]
\\
\nonumber
&=v(x)+\E\left[
\int_0^T e^{-\lambda t}
\left[(\cala^{(0)}-\lambda)[v]\big(X(t;x,u(\cdot))\big)
+\< L(u(t)),\,D^G v\big(X(t;x,u(\cdot))\big)\>_K\right]dt\right],
\\
\nonumber
&=v(x)+\E\left[
\int_0^T e^{-\lambda t}
\left[-F_0\left(X(t;x,u(\cdot)),\,D^G v\big(X(t;x,u(\cdot))\big)\right)+\,\< L(u(t)), D^Gv\big(X(t;x,u(\cdot))\big)\>_K\right]dt\right].
\end{align}
%where, in the last equality, we used \myref{strict}.
Since
$l$ is measurable and bounded from below by \myref{eq4:hpl},
the term $\E\big[\int_0^T e^{-\lambda t}l\big(X(t;x,u(\cdot)),u(t)\big)\,dt\big]$
is well defined, possibly equal to $+\infty$.
However,  \eqref{eq:finite} actually entails
%$\int_0^{+\infty} e^{-\lambda t} l(X(t;x,u(\cdot)),u(t))dt<+\infty$,
%then it must be also
$$
\E\left[\int_0^T e^{-\lambda t}l\big(X(t;x,u(\cdot)),u(t)\big)\,dt\right]<  \infty \ \ \ \forall T>0.
$$
Then, we can  add and subtract
$\E\big[\int_0^T e^{-\lambda t}l\big(X(t;x,u(\cdot)),u(t)\big)dt\big]$ in \myref{Dynkinv}
and use \myref{F0CV} to get, rearranging the terms,
\begin{align}\label{Dynkinvbis}
&\E\left[e^{-\lambda T}v\big(X(T;x,u(\cdot))\big)\right]
-v(x)+ \E\left[\int_0^T e^{-\lambda t}l\big(X(t;x,u(\cdot)),u(t)\big)\,dt\right]
\\
\nonumber
&=\E\left[
\int_0^T e^{-\lambda t}
\left[-F_0\left(X(t;x,u(\cdot)),\,D^G v\big(X(t;x,u(\cdot))\big)\right)+
F_{0,CV}\left(X(t;x,u(\cdot)),\, D^Gv\big(X(t;x,u(\cdot))\big); u(t)\right)\right]\,dt\right].
\end{align}
Now we let $T\to + \infty$. The right hand side has a limit (possibly $+\infty$), as the integrand is positive. The left hand side clearly converges to
$J(x;u(\cdot))- v(x)$. This implies that also the limit of the right hand side is finite and
\begin{align*}
&J(x;u(\cdot))- v(x)\\
&=\E\left[\int_0^{\infty} e^{-\lambda t}
\left[-F_0\left(X(t;x,u(\cdot)),\,D^G v\big(X(t;x,u(\cdot))\big)\right)+
F_{0,CV}\left(X(t;x,u(\cdot)),\, D^Gv\big(X(t;x,u(\cdot))\big); u(t)\right)\right]dt\right].
\end{align*}
The claim follows rearranging the terms.
%
%It follows from standard arguments by the abstract Dynkin formula (Theorem \ref{teo:dynkin}) and by  Proposition \ref{prop:strict}.
 \end{proof}

\begin{Theorem}[Verification theorem]
\label{th:ver}
Let \myref{eq4:hpl} hold. Let $v$ be a mild solution to (\ref{eq4:HJBcontrolellbis}) and assume that $v\in \mathcal{S}^{A,G}(H)$.  We have the following.
\begin{itemize}
\item[(i)]
 $v\leq V$ over $H$.
 \item[(ii)]
Let $x\in H$ and assume that there exists  $u^*(\cdot)\in \calu$ such that
$\P\times dt -\mbox{a.e.}$
\begin{equation}\label{eqverfi}
F_0\big(X(t;x,u^*(\cdot)),\, D^Gv\big(X(t;x,u^*(\cdot))\big)\big)= F_{0,CV}\big(X(t;x,u^*(\cdot)),\, D^Gv\big(X(t;x,u^*(\cdot))\big); u^*(t)\big).
\end{equation}
Then $v(x)=V(x)=J(x;u^*(\cdot))$.
 \end{itemize}
 \end{Theorem}
 \begin{proof}
\emph{(i)} By \myref{FI}, for all $u(\cdot)\in \calu$ such that
\eqref{eq:finite} holds, we have $v(x)\le J(x;u(\cdot))$, which yields this claim.

\emph{(ii)}
Let $u^*(\cdot)$ such that \myref{eqverfi} holds.
If $J(x;u^*(\cdot))<+\infty$, then, from \myref{FI},
we immediately get $v(x)= J(x;u^*(\cdot))$, which, combined with item (i), yields the claim.
We now prove that it cannot be $J(x;u^*(\cdot))=+\infty$.
Assume, by contradiction, that $J(x;u^*(\cdot))=+\infty$.  Then, by \myref{eqverfi}, we have $\P\times dt -\mbox{a.e.}$
\begin{equation}
\label{eq:lustar}
l\big(X(t;x,u^*(\cdot)),u^*(t)\big)
=
F_0\left(X(t;x,u^*(\cdot)),\, D^Gv\big(X(t;x,u^*(\cdot))\big)\right)
- \<L(u(t)),D^Gv\big(X(t;x,u^*(\cdot))\big)\>_K.
\end{equation}
{By \myref{strict}, $F_0(\cdot,D^Gv(\cdot))$ is bounded. Hence,   Assumption \ref{hp4:ABQforOU}-(iv),  the fact that $u^*(\cdot)\in \calu$ and \myref{eq:lustar} imply  $\E\big[\int_0^T e^{-\lambda t}l\big(X(t;x,u^*(\cdot)),u^*(t)\big)\,dt \big]<  \infty$ for all $T>0$.}
Then, we can argue as in the proof of Proposition \ref{prop:FI}
getting \myref{Dynkinvbis}  with $u^*(\cdot)$ in this case and, using again \myref{eqverfi},
\begin{align}\label{Dynkinvter}
&\E\left[e^{-\lambda T}v\big(X(T;x,u^*(\cdot))\big)\right]
-v(x)+ \E\left[\int_0^T e^{-\lambda t}l\big(X(t;x,u^*(\cdot)),u^*(t)\big)dt\right]=0.
\end{align}
Letting $T\to +\infty$ we get
$v(x)=J(x;u^*(\cdot))=+\infty$, a contradiction, as $v$ is finite.
\end{proof}

\subsection{Optimal feedback controls}
\label{SS:FEEDBACKS}

As usual, the verification theorem is composed of two statements: the first one states that the solution to the HJB equation enjoys the property of being smaller than the value function; the second one is the most important from the point of view of the control problem, as it furnishes a sufficient condition of optimality (\eqref{eqverfi} in our case). Then,  the problem becomes
the so-called synthesis of an optimal control, i.e. to produce a control $u^*(\cdot)$ verifying such condition. The answer relies in the study of the \emph{closed loop equation}.
%In our context, assume that

Let $v$ be a mild solution to HJB equation \eqref{eq4:HJBcontrolellbis}.
Assuming that the infimum of the map
 $$\Lambda \to \R , \ \ u\mapsto F_{0,CV}\,\big(x ,D^G v(x);u\big)$$ is attained
and defining
%in \eqref{eq2:feedbackregular},
the multivalued function {(\emph{feedback map})}
\begin{equation}
\label{eq4:feedbackregularell}
\begin{array}{l}
\Phi\colon H \ \longrightarrow \ 2^\Lambda,\\[5pt]
 \ \ \ \ \ \ \ x \ {\longmapsto} \ \arg\min_{u\in \Lambda}\,\,
F_{0,CV}\,\big(x ,D^G v(x);u\big),
\end{array}
\end{equation}
the closed loop equation (CLE) associated with our problem and to $v$ is indeed a stochastic differential inclusion:
\begin{equation}
\label{eq4:CLEinclusionell}
%\left \{
%\begin{array}{l}
{d} X(s) \in \big[AX(s)+ G L(\Phi (X(s)))\big]\,{d} s
+\sigma\, {d} W(s).
%\\[5pt]
%X(0)=x.
%\end{array}
%\right.
\end{equation}
%Similarly to Corollary \ref{cr2:optimalfeedbackinfhor}
We have the following result.
%whose proof is omitted as it is completely similar
%to the one of Corollary \ref{cr4:optimalfeedback}.

\begin{Corollary}
%[Optimal Feedback Controls]
\label{cr4:optimalfeedbackell}
%Let the assumptions of Theorem \ref{th:ver} hold.
Let \myref{eq4:hpl} hold. Let $v$ be a mild solution to (\ref{eq4:HJBcontrolellbis}) and assume that $v\in \mathcal{S}^{A,G}(H)$. Let $x\in H$ and  assume  that the feedback map $\Phi$
defined in (\ref{eq4:feedbackregularell}) admits a measurable selection
$\phi:H \to U$ and consider the SDE
\begin{equation}
\label{eq4:CLEselectionell}
\begin{cases}
{d} X(s) = \left[AX(s)+ GL(\phi (X(s))) \right]{d} s +
\sigma\, {d} W(s)\\[5pt]
X(0)=x.
\end{cases}
\end{equation}
Assume that \eqref{eq4:CLEselectionell} has a mild solution {in $\mathcal{M}_\calp^{1,loc}(U)$, i.e. there exists $X_\phi(s;x)\in \mathcal{M}_\calp^{1,loc}(U)$ such that}
\begin{equation}\label{eq4:CLEselectionellbis}
X_\phi(t;x):=e^{tA}x+ \int_0^t e^{(t-s)A}\sigma\, dW(s)
+\int_0^t \overline{e^{(t-s)A} G}L(\phi(X_\phi(s;x)))\,ds, \ \ \ \forall t\geq 0.
\end{equation}
 Define, for $s \ge 0$, $u_\phi (s):=\phi(X_\phi(s;x))$ and assume that $u_\phi(\cdot)\in \calu$. Then $v(x)=V(x)=J(x;u_\phi(\cdot))$. In particular  the couple $({u_\phi}(\cdot),{X_\phi}(\cdot;x))$
is optimal at $x$.

Moreover, if $\Phi(x)$ is single-valued and the mild solution
to \eqref{eq4:CLEselectionell} is unique, then the optimal control is unique. %{\blu{RIFLETTERE SU QUESTO: MI PARE CHE LO SIA LO STATO OTTIMO; PER IL CONTROLLO HO DEI DUBBI}}
%\end{enumerate}
\end{Corollary}
\begin{proof}
 Consider the  couple $(u_{\phi}(\cdot),X_{\phi}(\cdot))$ and observe that $X_{\phi}(\cdot)$ is the unique mild solution (in the strong probabilistic sense) of the state equation associated to the control $u_{\phi}(\cdot)$, so that $X_\phi(\cdot;x)\equiv X(\cdot;x,u_\phi(\cdot))$.
By construction such couple satisfies \myref{eqverfi}.
Then, by Theorem \ref{th:ver}-(ii) we obtain that it   is optimal.

Let us address now the  uniqueness issue. We observe that, if
$(\hat u(\cdot), X(\cdot;x,\hat{u}(\cdot)))$ is another optimal couple at $x$, we immediately have, by \eqref{FI} and the fact that $v(x) =V(x)$,
$$
\E\left[\int_{0}^{\infty} e^{-\lambda s}\left[
		F_{0}\left(X(s;x,\hat{u}(\cdot)),\,D^G v\big( X(s;x,\hat{u}(\cdot))\big)\right) -
		F_{0,CV}\left(X(s;x,\hat{u}(\cdot)),\,D^G  v\big(X(s;x,\hat{u}(\cdot))\big);\hat u(s)\right)
		\right]ds\right]=0.
$$
As the integrand is always negative and as $\Phi$ is single-valued, this implies that  $\P\times ds$-a.e. we have
$\hat u(\cdot)=\Phi\big(X(\cdot;x,\hat{u}(\cdot))\big)$. This shows that  $X(\cdot;x,\hat{u}(\cdot))$ solves \eqref{eq4:CLEselectionell}. Then uniqueness of mild solutions to \eqref{eq4:CLEselectionell} gives the claim.
\end{proof}
%{\red{\begin{Remark}\label{rem:specific}
%To prove the results of this section, Assumption \ref{ass:G} is only needed to show that the conlcusion of  Lemma \ref{lemma1} holds when $f$ is the  mild solution $v$ of the HJB equation associated to our initial control problem. This could also be done in different contexts, where the actual Assumption \ref{ass:G} does not hold.
%In particular, we mention the stochastic control problem with delay in the  control variable studied in \cite{GozziMasiero16a} (see Section \ref{sec:delay} below). For such problem
%Assumption \ref{ass:G} does not hold (taking as $G$ the bounded operator $B$ described in \cite[Subsection 3.2]{GozziMasiero16a}). However, it can be proved, using the arguments of \cite[Subsection 4.2]{GozziMasiero16a}, that the statement of Lemma \ref{lemma1} holds in this case for the class of functions $\Sigma^1_{T,\alpha}$
%defined in \cite[Def.\,5.1]{GozziMasiero16a}, to which the mild solution $v$ of the associated HJB equation belongs.
%\end{Remark}}}
%
%

We conclude the section commenting on the extension of  our results to the case when the control problem is considered in the so-called \emph{weak formulation}. So far, we have considered our family of stochastic optimal control problems in the {\em strong formulation}.
% for the stochastic optimal control problem.
It is possible to consider the problem also in the so-called  weak formulation, i.e. letting the filtered probability space and the Wiener process vary with the control strategy $u(\cdot)$
(see, e.g., \cite[Ch.\,2]{YongZhou99}).
More precisely, in the weak formulation, the control strategy is a $6$-tuple
$
\left(\overline{\Omega},\overline{\mathcal{F}},\{\overline{\mathcal{F}_t}\}_{t\geq 0}, \overline{\P},  \overline{W}, \overline{u}(\cdot) \right)
$.
%where $\left(\overline{\Omega},\overline{\mathcal{F}},\{\overline{\mathcal{F}_t}\}_{t\geq 0}, \overline{\P}\right)$ is a filtered probability space, $\overline{W}=(\overline{W_t})_{t\geq 0}$ is a $\Xi$-valued cylindrical Brownian motion, and $\overline{u}(\cdot):\overline{\Omega}\times[0,+\infty)\to \Lambda$ is $\{\overline{\mathcal{F}_t}\}_{t\geq 0}$-progressively measurable with
%$\int_0^t \E|\overline{u}(s)|_U^{p}ds <+\infty$ for all $t\geq 0$.
Calling $\overline\calu$ the set such control strategies, the objective is to minimize the cost \myref{eq4:CF1ell} over $\overline\calu$. The resulting value function $\overline V$ is, in principle, smaller than $V$. The main advantage in choosing such formulation is  that existence of optimal control strategies in feedback form is easier to obtain.
The verification theorem above also holds when we consider the control problem in its weak formulation.
Indeed, the proof of Theorem \ref{th:ver} works for every filtered probability space and any cylindrical Brownian motion  on it. Hence, letting the filtered probability space and the cylindrical Brownian motion  vary,
one gets that $v \le \overline V$ over $H$. Moreover, if \myref{eqverfi} holds for a given control strategy\,(\footnote{Elements of $\overline\calu$ are, rigorously speaking, $6$-tuples; however, for simplicity, we denote them simply by $\overline{u(\cdot)}$.}) $\overline{u^*}(\cdot)\in \overline \calu$, then we have $v(x)=\overline V(x)=J(x;\overline{u^*}(\cdot))$.
One gets the following.
\begin{Corollary}
%[Optimal Feedback Controls]
\label{cr4:optimalfeedbackellbis}
Let \myref{eq4:hpl} hold. Let $v$ be a mild solution to (\ref{eq4:HJBcontrolellbis}) and assume that $v\in \mathcal{S}^{A,G}(H)$. Let $x\in H$ and  assume  that the feedback map $\Phi$
defined in (\ref{eq4:feedbackregularell}) admits a measurable selection
$\phi:H \to U$. Assume now that  \myref{eq4:CLEselectionell} has a \emph{martingale solution}\footnote{Weak-mild solution in the terminology of \cite{FGSbook}.}  (see \cite[p.\,220]{DaPratoZabczyk14} or \cite[Def.\,3.1,\,p.\,75]{GawareckiMandrekar10} for the definition)
$\overline{X_\phi}(\cdot;x)$ in some filtered probability space
$\left(\overline{\Omega},\overline{\mathcal{F}},\big\{\overline{\mathcal{F}_t}\big\}_{t\geq 0},\, \overline{\P} \right)$
and for some $\Xi$-valued cylindrical Brownian motion $\overline{W}$ defined on it.
Define, for $s \ge 0$, $\overline{u_\phi} (s)=\phi\big(\overline{X_\phi}(s;x)\big)$
and assume  $\overline{u_\phi}(\cdot)\in \overline\calu$\,(\footnote{In the sense that the $6$-tuple identified by $u_\phi$ belongs to $\overline \calu$.}).
Then $v(x)=\overline{V}(x)=J(x;\overline{u_\phi}(\cdot))$. In particular   $\big(\overline{u_\phi}(\cdot),\overline{X_\phi}(\cdot;x)\big)$
is an optimal couple.
\end{Corollary}

\section{Applications}\label{sec:app}

{In the present section we provide two examples of application of our results.

The first
example, fully developed, concerns the optimal control of the stochastic heat equation in a given space region ${\mathcal{O}}\subseteq \R^d$ when the control can be exercised only at the boundary  $\partial{\mathcal{O}}$.
Precisely, we consider the case when the control at the boundary enters through a Neumann-type boundary condition, corresponding to control the heat flow at the boundary.} The existence and uniqueness of mild solutions to the associated elliptic HJB equation in this case is guaranteed (under suitable conditions) by the results of \cite{FedericoGozzi16}.

The second example
concerns the optimal control of a stochastic differential equation with delay in the control process (see \cite{GozziMasiero12,GozziMasiero16b} for the treatment of the same problem over finite horizon). In this case, the result we give needs to \emph{assume} the existence of a mild solution to the associated elliptic HJB equation. The reason for that is that a theory of mild solutions for  elliptic HJB equations associated to this kind problem has not been yet developed in the elliptic case. Indeed, unlike the first example, this kind of  equations is not covered  by the results of \cite{FedericoGozzi16}, due to the lack of $G$-smoothing. In this case it is needed an \emph{ad hoc} treatment of the equation, dealing with the specific case at hand, to show the existence of mild solutions (see, e.g.,  the aforementioned references \cite{GozziMasiero12,GozziMasiero16b} in the parabolic case). Although a result of this kind for elliptic equation seems straightforward, a rigorous statement of this result has not been rigourously fixed yet.  For this reason, we limit ourselves to provide a weaker result taking the existence of mild solutions to the associated HJB equation as an assumption and leaving the investigation of that for future work.
Due to the lack of a rigourous background
on which relying our results, we do not state in this case a theorem and just keep the arguments at the level of an informal exposition.
\subsection{Neumann Boundary control of a stochastic heat equation with additive noise}
\label{SS:NEUMANN}
We consider  the optimal control of a nonlinear stochastic heat equation in a given space region ${\mathcal{O}}\subseteq \R^d$ when the control can be exercised only at the boundary of  ${\mathcal{O}}$.
%% or in a subset of ${\mathcal{O}}$.
%Precisely we consider the cases when the control at the boundary enters through a Neumann-type boundary condition, corresponding to control the heat flow at the boundary.

\subsubsection{Problem setup}
\label{SSSE2:HEATBOUNDARYSETTINGD}
Let ${\mathcal{O}}$ be an open, connected, bounded subset of  ${\mathbb{R} }^d$ with regular (in the sense of \cite[Sec.\,6]{Lasiecka80}) boundary $\partial {\mathcal{O}}${\footnote{{We stress that such conditions may allow corners in the boundary: in particular, when $d=2$ squares satisfy the required regularity.}}}. We consider the controlled dynamical system driven by the following SPDE in the time interval $[0,+\infty)$:
%\begin{equation}
%\label{eq2:state-boundaryPDEDir}
%\left\{
%\begin{array}{ll}
%\displaystyle{dy(s,\xi) =
%\frac{\partial^2}{\partial \xi^2} y(s,\xi)
% ds+[\sigma dW (s)](\xi)},
%& (s,\xi) \in [0,+\infty) \times {\mathcal{O}},
%\\\\
%y(0,\xi ) = x (\xi) & \xi\in {\mathcal{O}}, \\\\
%\displaystyle{ y(s,\xi) = \alpha (s,\xi),}
%& (s,\xi)\in  [0,+\infty)\times \partial {\mathcal{O}}, \end{array}\right.
%\end{equation}
%\begin{equation}
%\label{eq2:state-boundaryPDEDir}
%\left\{
%\begin{array}{ll}
%\displaystyle{\frac{\partial}{\partial s} y(s,\xi) =
%\frac{\partial^2}{\partial \xi^2} y(s,\xi)+ \dot{W}_Q(s\xi)},
%& (s,\xi) \in [0,+\infty) \times {\mathcal{O}},
%\\\\
%y(0,\xi ) = x (\xi) & \xi\in {\mathcal{O}}, \\\\
%\displaystyle{ y(s,\xi) = \alpha (s,\xi),}
%& (s,\xi)\in  [0,+\infty)\times \partial {\mathcal{O}}, \end{array}\right.
%\end{equation}
\begin{equation}
\label{eq2:state-boundaryPDENeu}
\left\{
\begin{array}{ll}
\displaystyle{\frac{\partial y(t,\xi)}{\partial t} =
\Delta y(t,\xi)+ \sigma \dot{W}(t,\xi)},
& (t,\xi) \in [0,+\infty) \times {\mathcal{O}},
\\\\
y(0,\xi ) = x (\xi), & \xi\in {\overline{\mathcal{O}}}, \\\\
\displaystyle{ \frac{\partial y(t,\xi)}{\partial n}= {\gamma_0} (t,\xi),}
& (t,\xi)\in  [0,+\infty)\times \partial {\mathcal{O}}, \end{array}\right.
\end{equation}
where:
\vspace{-.2cm}
\begin{itemize}
   \item $y: [0,+\infty)\times \mathcal{O} \times \Omega\to \R$ is the stochastic process describing  the evolution of the temperature distribution and is the {\em state variable} of the system;
   \vspace{-.2cm}
   \item ${\gamma_0}:[0,+\infty)\times \partial \mathcal{O}\times \Omega\to \R$ is the stochastic process representing the heat flow at the boundary; it is the {\em control variable} of the system and acts at the boundary of it: this is the reason of the terminology ``boundary control";
   \vspace{-.2cm}
%  \item $\Delta_\xi$ is the Laplace operator;
%  \item $f \in C^0(\mathbb{R})$ is a nonlinear function of the state, which may represent a ``reaction'' term;
\item $n$ is the outward unit normal vector at the boundary $\partial \calo$;
\vspace{-.2cm}
  \item $x\in L^2(\mathcal{O}) $ is the initial state (initial temperature distribution) in the region $\mathcal{O}$;
  \vspace{-.2cm}
  \item $W$ is a cylindrical Wiener process in $L^2(\calo)$;
  \vspace{-.2cm}
  \item $\sigma\in \call(L^2(\calo))$.
%   \item $\sigma\in \call_2(\Xi;L^2(\calo))$ is a diffusion operator.

\end{itemize}

%\smallskip
Assume that this equation is well posed (in some suitable sense,
see below for the precise setting) for every given ${\gamma_0(\cdot,\cdot)}$ in a suitable set of admissible control processes 
%\footnote{For the concept of solution and the the assumptions on the data $f, x, Q$ and on the control strategy $\alpha $ that guarantee the existence and uniqueness of it %see next section.}
and denote its unique solution  by ${y^{x,\gamma_0(\cdot,\cdot)}}$ to underline the dependence of the state $y$ on the control ${\gamma_0(\cdot,\cdot)}$ and on the initial datum $x$.
The controller aims at minimizing, over the setof admissible controls, the {objective} functional
\begin{equation}
\label{eq1:CFLapDir}
{\mathbb{E}}\Bigg[ \int_{0}^{\infty}e^{-\lambda t}\left(\int_{{\mathcal{O}}}\ell_1\big(y^{x,{\gamma_0(\cdot,\cdot)}}(t,\xi )\big) \,d\xi + \int_{{\partial\mathcal{O}}} \! \ell_2 \big({\gamma_0} (t,\xi)\big)\, d\xi \right)\,\ud t   \Bigg],
\end{equation}
{where $\ell_1, \ell_2:\R\to\R$ are given measurable functions bounded from below and $\lambda>0$ is a discount factor.}

\subsubsection{Infinite dimensional setting}
\label{SSSE2:HEATBOUNDARYINFDIMSETTING}

We now rewrite the state equation  (\ref{eq2:state-boundaryPDENeu}) and the functional \eqref{eq1:CFLapDir} in an infinite dimensional setting in the space  $H\coloneqq L^2({\mathcal{O}})$. For more details, we refer to \cite[Sec.\,5]{FedericoGozzi16} and references therein.
Consider the realization of the Laplace operator with vanishing Neumann boundary conditions\,\footnote{To be precise, $\cald(A_N)$ is the closure in $H^2(\calo)$ of the set of functions $\phi\in C^2(\overline\calo)$ having vanishing normal derivative at the boundary $\partial\calo$.}:
\begin{equation}
\label{eq:def-AN}
\left \{
\begin{array}{l}
\cald(A_N) := \left \{ \phi\in H^2 ({\mathcal{O}})  : \  \frac{\partial \phi}{\partial n} = 0 \; \text{\rm\;  on } \partial{\mathcal{O}}  \right \}, \\[7pt]
A_N\phi:= \Delta\phi, \ \ \forall \phi\in \cald(A_N).
\end{array}
\right.
\end{equation}
It is well-known (see, e.g.,  \cite[Ch.\,3]{Lunardi95}) that $A_N$ generates a strongly continuous  analytic  semigroup $\big\{e^{t A_N}\big\}_{t\geq 0}$ in $H$. Moreover, $A_N$ is a self-adjoint and dissipative operator. In particular $(0,+\infty)\subset \varrho(A_N)$, where $ \varrho(A_N)$ denotes the resolvent set of $A_N$. So,  if $\delta > 0$, then  $(\delta I-A_N)$ is invertible and $(\delta I-A_N)^{-1}\in \mathcal{L}(H)$. {Moreover (see, e.g., \cite[App.\,B]{Lasiecka80})
the operator $(\delta I-A_N)^{-1}$ is compact. Consequently, there exists an orthonormal complete sequence $\{e_k\}_{k\in \N}$ such that the operator
$A_N$ is diagonal with respect to it:
\begin{equation}\label{diag}
A_Ne_k=-\mu_k e_k, \ \ k\in\N,
\end{equation}
for a suitable sequence of eigenvalues $\{\mu_k\}_{k\in \N}\subseteq \R^+$ repeated according to their multiplicity (they are nonnegative  due to dissipativity of $A_N$). We assume that such sequence is increasingly ordered. {Then,  $\mu_0=0$, as clearly the constant functions belong to Ker\,$(A_N)$, and  $\mu_k>0$ for each $k\in \N_0:=\N\setminus\{0\}$, since, as an immediate consequence of the Gauss-Green formula, only the constant functions belong to  Ker\,$(A_N)$.}
%By compactness of $(\delta-A_N)^{-1}$, we have  $\mu_k \to + \infty $ as $k \to +\infty$.
Moreover, \cite[Sec.\,5.6.2, p.\,395]{Triebel78}
(see also \cite[App.\,B]{Lasiecka80}) provides also a growth rate for the sequence of eigenvalues; indeed
\begin{equation}
\label{kj}
\mu_k \sim k^{2/d}.
\end{equation}
%Note that \eqref{kj} in particular yields \blu{FORSE $0$ E' SEMPRE SEMPLICE???}
%\begin{equation}\label{kj2}
%\exists k_0\in\N: \ \mu_k>0 \ \forall k\geq k_0.
%\end{equation}}
We have (see, e.g., \cite[App.\,B]{Lasiecka80}) the isomorphic identification
\begin{equation}\label{pppd}
\cald\big((\delta I-A_N)^\alpha\big)=H^{2\alpha}({\mathcal{O}}),\quad \forall \alpha\in \left(0,\frac{3}{4}\right), \  \forall \delta>0,
\end{equation}
where $H^{s}(\calo)$ denotes the Sobolev space of exponent $s\in \R$.
Next, consider the following problem with Neumann boundary condition:
\begin{equation}
\label{eq:Neumann-problem}
\left \{
\begin{array}{ll}
\Delta w(\xi)=\delta w(\xi), & \xi\in {\mathcal{O}} \\[8pt]
\frac{\partial w}{\partial n}(\xi) = \alpha(\xi), & \xi\in \partial{\mathcal{O}}.
\end{array}
\right .
\end{equation}
 Given any $\delta> 0$ and $\alpha\in L^2(\partial {\mathcal{O}})$, there exists a unique solution $N_\delta\alpha\in H^{3/2}({\mathcal{O}})$ to (\ref{eq:Neumann-problem}). Moreover, the operator (\emph{Neumann map})
\begin{equation}
\label{eqapp:defNeumannmap}
N_\delta :  L^2(\partial {\mathcal{O}}) \to H^{3/2}(\calo),
\end{equation}
is continuous (see \cite[Th.\,7.4]{LionsMagenes-1-EN}).
So, in view of \eqref{pppd}, the map
\begin{equation}
\label{eqapp:defNeumannmapdelta}
N_\delta :  L^2(\partial {\mathcal{O}}) \to
\cald\big((\delta I -A_N)^{\frac{3}{4} - \varepsilon}\big), \ \ \ \varepsilon\in(0,3/4),
\end{equation}
is continuous. In \cite[Sec.\,5]{FedericoGozzi16}, it is shown that the natural abstract reformulation of the original control problem in the space $H$ is
\begin{equation}
\label{eq:strongformAN}
\left \{
\begin{array}{l}
\ud X(t) = \left [  A_N X(t) + G_N^{\delta,\varepsilon} L_N^{\delta,\varepsilon}\gamma(t) \right ] \,d t
+ \sigma \,d W(t), \\[8pt]
X(0) = x.
\end{array}
\right.
\end{equation}
where
$
L_N^{\delta,\varepsilon}:= (\delta I-A_N)^{\frac{3}{4}-\varepsilon} N_\delta \in \mathcal{L}(L^2(\partial\calo); H)$,
 $G_N^{\delta,\varepsilon}\coloneqq  (\delta I-A_N)^{\frac{1}{4}+\varepsilon}$, and  $u(t)\coloneqq \gamma_0(t,\cdot)\in L^2(\partial \calo)$ for $t\geq 0$.
We are now in the framework of \eqref{eq4:SE1ellbis}, with $K=H$, $A=A_N$, $G= G_N^{\delta,\varepsilon}$, $L= L_N^{\delta,\varepsilon}$, and $U=L^2(\partial\calo)$.
Let us consider, as  set of admissible controls,
$$
\calu\coloneqq \left\{u: [0,+\infty)\times \Omega \to \Lambda: \hbox{ $u(\cdot)$ is
$\big\{\mathscr{F}_t\big\}_{t\geq 0}$-prog. meas. and s.t.} \
\int_0^t \E\left[|u(s)|_{L_2(\partial\calo)}^{p}\right]\,ds <\infty \ \ \forall t\geq0 \right\},
$$
where $\Lambda\subseteq L^2(\partial \calo)$ {and $p$ will be specified later} {according to \eqref{eq:choicep}}.
% $L_N^{\delta,\varepsilon}(\Lambda)$ is bounded in $H$
%and that,
 Defining
 $$ l_1(x) \coloneqq \int_{{\mathcal{O}}}\ell_1 (x(\xi)) \ud\xi, \ \ \ \ \ l_2(u)\coloneqq \int_{{\partial\mathcal{O}}}\ell_2 (u(\xi )) \ud\xi,
 $$
and
\[
l\colon H\times \Lambda \to \mathbb{R},  \ \ \ l(x,u)\coloneqq l_1(x)+l_2(u),
\]
the functional (\ref{eq1:CFLapDir}) can be rewritten in the Hilbert space framework as
\begin{equation}
\label{eq1:CFLapDir-restatedbis}
J(x;u(\cdot)) \coloneqq{\mathbb{E}}\Bigg[ \int_{0}^{\infty} e^{-\lambda t}
l \big(X(t;x,u(\cdot)), u(t)\big)\, dt \Bigg].
\end{equation}

\subsubsection{HJB equation and verification theorem}
\label{SSSE2:HEATBOUNDARYHJBVT}

Setting $Q\coloneqq \sigma\sigma^*$, the HJB equation associated to the minimization of \eqref{eq1:CFLapDir-restatedbis} is
\begin{equation}
\label{eq2:HJDirichletNeumann}
\displaystyle{\lambda v(x) - \frac{1}{2}\;
\mbox{\rm Tr}\;[QD^2v(x)]
- \left\langle A_Nx,Dv(x) \right\rangle_H- l_1(x)- \inf_{u \in \Lambda} \left\{\left\langle  L_N^{\delta, \varepsilon}  u,D^{G_N^{\delta,\varepsilon}} v(x) \right\rangle_H +l_2(u)  \right\}=0.}
\end{equation}
Since the semigroup $\big\{e^{tA_N}\big\}_{t \ge 0}$ is strongly
continuous and analytic, then by  \cite[Th.\,6.13(c)]{Pazy83} the operator
$e^{tA_N} G_N^{\delta,\varepsilon}$ can be extended to  $\overline{e^{tA_N} G_N^{\delta,\varepsilon}}=G_N^{\delta,\varepsilon}e^{tA_N}\in \call(H)$ for every $t>0$ and
\begin{equation}\label{Assq}
\left|\overline{e^{tA_N} G_N^{\delta,\varepsilon}}\right|_{\call(H)}\leq
Ct^{\frac14 + \varepsilon}, \qquad \forall t > 0.
\end{equation}
Hence, Assumption \ref{hp4:ABQforOU}(i) and (iii) is satisfied with $A=A_N$, $G=G_N^{\delta,\varepsilon}$, and $\beta=\eps+1/4$. {Consequently, recalling \eqref{eq:choicep}, we choose $p>\frac{1}{\frac{3}{4}-\eps}$.}
\smallskip

Now, assume the following.
{\begin{itemize}
\item[(H1)] {$\sigma$ satisfies} Assumption \ref{hp4:ABQforOU}(ii).
\item[(H2)] Conditions (A1) and (A2) of Remark \ref{rem:exist} hold true with $G=G_N^{\delta,\varepsilon}$.
 \item[(H3)] {$\ell_1\in C_b(\R)$, so }$l_1\in C_b(H)$\footnote{{According to  Remark \ref{rm:polgrowth} it is possible to deal with the case when $\ell_1$, and so $l_1$, has polynomial growth.}}. Moreover the map $q\mapsto F_1(q)$, defined by
  $$F_1(q):= \inf_{u \in \Lambda} \left\{\left\langle  L_N^{\delta, \varepsilon} u, q \right\rangle_H +l_2(u)  \right\}, \ \ \ \ \  q\in H,$$ is Lipschitz continuous. These conditions imply that $F_0(x,q)=l_1(x)+F_1(q)$ satisfies condition (A3) of  Remark \ref{rem:exist}.
%  we have
% \begin{equation}\label{asss}
%\overline{e^{sA_N}G_N^{\delta,\varepsilon}}(H)\subseteq Q_{s}^{1/2}(H), \qquad \forall s > 0.
%\end{equation}
%\item[(A3)] The map $t\mapsto |\Gamma_{G_N^{\delta,\varepsilon}}(t)|_{\call(K,H)}$ belongs to $\cali$.
\end{itemize}}
%Under such assumptions, Assumption \ref{hp4:b1VTell} is verified; it follows that,   for every $\gamma\in \calu$, there exists a unique mild solution  $X(\cdot,x,\gamma(\cdot))\in \calh_\calp^{p,loc}(H)$ to \eqref{eq:strongformAN} for every $p\geq 2$.  Defining
%\[
%l\colon H\times \Lambda \to \mathbb{R},  \ \ l(x,\alpha) \coloneqq \int_{{\mathcal{O}}}\beta_1 (x(\xi)) \ud\xi+\int_{{\partial\mathcal{O}}}\beta_2 (\alpha(\xi )) \ud\xi,
%\]
%the functional  (\ref{eq1:CFLapDir}) can be rewritten in the Hilbert space setting as
%\begin{equation}
%\label{eq1:CFLapDir-restatedbis}
%J(x;a(\cdot)) ={\mathbb{E}}\Bigg[ \int_{0}^{+\infty} e^{-\lambda s}
%l (X(s;x,a(\cdot)), \gamma(s)) ds \Bigg].
%\end{equation}
%
%Hence, if $\beta_1,\beta_2$ satisfy proper continuity and growth assumptions guaranteeing that
%$$F_0(x,y,z)\coloneqq   \inf_{a \in \Lambda} \left\{\left\langle  B a,z \right\rangle +l(x,a)  \right\}$$  satisfies Assumptions \ref{hp4:F0ellbis}, we can apply Corollary \ref{corcor} to this problem.
Then, under such assumptions,  by Remark \ref{rem:exist}, for sufficiently large $\lambda >0$ there exists a unique mild solution $v$ to \eqref{eq2:HJDirichletNeumann}. By definition of mild solution, we have  $v\in C^{1,G}_b(H)$. Furthermore,  Assumption \ref{ass:G} is verified through Remark \ref{rem:imppo} in this case. Hence Proposition \ref{lemma1} applies yielding $v\in\cals^{A,G}(H)$ and enabling the application of Theorem \ref{th:ver}. We now discuss the validity of the above assumptions (H1)--(H3).
\begin{itemize}
\item
\emph{On the validity of (H1).}
First of all, we note that in Assumption \ref{hp4:ABQforOU}(ii), we can take $\gamma$ as small as we want; indeed, if this assumption holds true for some $\bar\gamma\in (0,1/2)$, then it holds true also for all $\gamma\in (0,\bar\gamma)$.
By \eqref{diag}, the operator $e^{tA_N}$ is diagonal with respect to the orthonormal basis $\{e_k\}$ with eigenvalues $e^{-t\mu_k}$. Assumption \ref{hp4:ABQforOU}(ii) rewrites as
\begin{equation}\label{cond1}
\int_0^t\left(s^{-2\gamma}\sum_{k\in\N} \left\langle e^{sA}Qe^{sA^*}e_k,e_k\right\rangle_H \right)ds=
\int_0^t\left(s^{-2\gamma} \sum_{k\in \N}e^{-2 \mu_k s}|\sigma e_k|_H^2\right)ds<\infty \ \ \forall t\geq 0.
\end{equation}
Applying Fubini-Tonelli's Theorem and considering \eqref{kj}
% one gets that \eqref{cond1} holds true if and only if
%\begin{equation*}
%\sum_{k\in\N}\int_0^t\left(s^{-2\gamma} e^{-2 \mu_k s}|\sigma e_k|^2\right)ds<+\infty\ \ \ \ \forall t\geq 0,
%\end{equation*}
%{i.e. if and only if
%\begin{equation}\label{cond2}
%\sum_{k\in\N_0}\int_0^t\left(s^{-2\gamma} e^{-2 \mu_k s}|\sigma e_k|^2\right)ds<+\infty \ \ \ \ \forall t\geq 0.
%\end{equation}}
%Now, calling $\Gamma_E$ the Euler's gamma function, i.e.
%$$
%\Gamma_E(\alpha)\coloneqq \int_0^{+\infty} s^{\alpha-1} e^{-s}ds, \ \ \alpha>0,
%$$
%we estimate
%\begin{equation}\label{cond2bis}
%\sum_{{k\in \N_0}}\int_0^t\left(s^{-2\gamma} e^{-2 \mu_k s}|\sigma e_k|^2\right)ds \leq  \Gamma_E(1-2 \gamma)\sum_{{k\in \N_0}} \mu_k^{2\gamma-1}|\sigma e_k|^2.
%\end{equation}
%%with equality when taking the  limit for $t\to+\infty$.
%So, considering \eqref{kj},
we see that \eqref{cond1} holds if
\begin{equation}\label{cond4}
\sum_{{k\in\N_0}}k^{\frac{2(2\gamma-1)}{d}}|\sigma e_k|_H^2<\infty.
\end{equation}
Let $\theta\geq 0$ be such that
\begin{equation}\label{limsup}
\limsup_{k\to\infty}\frac{|\sigma e_k|_H^{2}}{k^{-{2}\theta}}<\infty
\end{equation}
(recall that  $\sigma \in \mathcal{L}(H)$, so $\theta=0$ always verifies \eqref{limsup}).
Considering that $\gamma$ can be taken as small as we want and combining \eqref{cond4} and \eqref{limsup},
we conclude that (H1) holds if we may take in \eqref{limsup}
\begin{equation}\label{limsupbis}
\theta>\frac{1}{2}-\frac{1}{d}.
\end{equation}
In particular,
if $d=1$, then (H1) holds true for all $\sigma \in \mathcal{L}(H)$.
\item
\emph{On the validity of (H2).}
By \eqref{kj}, we have, for $k\in\N$,
$$G_N^{\delta,\varepsilon}e_k =\big(\delta I-A_N\big)^{\frac{1}{4}+\varepsilon} e_k= g_k e_k, \ \ \ \mbox{where} \ g_k \coloneqq\left(\delta+\mu_k\right)^{\frac{1}{4}+\varepsilon.}$$
 The operator $\overline{e^{tA_N}G_N^{\delta,\varepsilon}}$ is diagonal too with respect to $\{e_k\}_{k\in\N}$ and
 \begin{equation}\label{equG}
 \overline{e^{tA_N} G_N^{\delta,\varepsilon}}e_k=e^{-\mu_k t} g_k e_k= e^{-\mu_k t} \left(\delta+\mu_k\right)^{\frac{1}{4}+\varepsilon}e_k, \ \ k\in\N.
 \end{equation}
% It follows that Assumption  \ref{hp4:ABQforOU}(iii) is satisfied if and only if
%\begin{equation}\label{eq4:etAGextend}
%\sup_k \left| e^{-\mu_{k}t}\left(\delta+\mu_k\right)^{\frac{1}{4}+\varepsilon} \right|< + \infty, \ \ \ \forall t>0.
%\end{equation}
%\begin{equation}\label{esg}
% |\overline{e^{sA}G}|_{\call(K,H)}\le C_G[s^{-\beta}\vee 1]e^{a_Gs} \ \ \forall s>0.
%\end{equation}
%above we refer to \cite[Example\,5.3]{FedericoGozzi16}.
%Here we simply recall that to have both (H1) and (H2) satisfied we need to require $d<3$. In particular both (H1) and (H2) are satisfied when $d=1$, $\sigma=I$ and also when $d=2$ and $\sigma$ is a suitable diagonal operator with respect to $\{e_k\}$. } %such that $|\sigma e_k|_H \sim k^{-\eta}$, for suitable $\eta >0$, as $k \to + \infty$.
Assume now further that $\sigma$ is diagonal with respect to
 $\{e_{k}\}_{k\in\N}$ and nondegenerate, i.e.
 $\sigma e_k=\sigma_ke_k$ for every $k\in \N,
 $
 where $\sigma_k> 0$ for every $k\in\N$. Set $q_k\coloneqq \sigma_k^2>0$ for $k\in\N$.
%and assume that $A$, $Q$, and $G$ admit spectral
%decompositions
%$$
%Ae_{n}=-\alpha_{n}e_{n}, \qquad Q e_{n}=q_{n}e_{n},
%\qquad Ge_{n}=g_{n}e_{n}, \ \ \  \ \ \ \forall n\in\mathbb{N},
%$$
%where
%$\alpha_{n}\geq 0$, $g_n\in \R$,  $q_{n}>0$ for all $n\in\N$ and $\alpha_{n}\uparrow+\infty$ as $n\to\infty$.
Then
%$$
% e^{sA_N}Qe^{sA_N^*}e_k=e^{-2s\mu_k}q_ke_k, \ \ \ k\in\N, \ s>0.
%$$
%Note that,   the set $\Sigma_0\coloneqq \{n\in\N: \ \alpha_n=0\}$ is finite\,\footnote{In the quoted
%references  at the beginning of this example it is assumed
%$\alpha_n>0$. Here we let some $\alpha_n$ to be $0$ to cover the case
%% developed in Subsection 5.1.}. Call $\Sigma_0^c=\N\setminus \Sigma_0$.
%In this case
$Q_t$
 is diagonal too. Moreover and $Q_te_0=tq_0e_0$ and
% \red{under the agreement $\frac{1-e^{-2\mu_0t}}{2\mu_0}\coloneqq t$ (recall that $\mu_0=0$)}, we have
{$$
Q_t e_k=
\frac{q_k}{2\mu_k} (1-e^{-2\mu_k t})e_k, \ \ \ \mbox{if} \ k\in\N_0,
\ \ \ \ \ \forall t\geq 0.
$$}
Hence,  with the   agreement $\frac{1-e^{-2\mu_kt}}{2\mu_k}\coloneqq t$ if {$k=0$}, we have
$$
\Gamma_G(t)e_k\coloneqq Q_t^{-1/2}\overline{e^{tA_N}G_N^{\delta,\varepsilon}} e_k
=
\sqrt{
\frac{2\mu_{k}}{(1-e^{-2t\mu_{k}})q_k}
}\;\;
 e^{-\mu_k t} \left(\delta+\mu_k\right)^{\frac{1}{4}+\varepsilon}e_k
%=\sqrt{
%\frac{2\alpha_{k}}{e^{2t\alpha_{k}}-1}
%\cdot \frac{g_k^2}{q_k} }\;\;
%e_n,
\ \ \ \forall k\in\N.
$$
Since $|\Gamma_G(t)|_{\mathcal{L}(H)}=\sup_{k \in \N}\big|\Gamma_G(t)e_k\big|_H$, then,
% This implies that
%$$
%|\Gamma_G(t)|_{\mathcal{L}(H)}=\sup_{k \in \N}|\Gamma_G(t)e_k|=
%\sup_{k \in \N}\sqrt{
%\frac{2\mu_{k} \left(\delta+\mu_k\right)^{\frac{1}{2}+2\varepsilon}}{(e^{2t\mu_{k}}-1)
%q_k}}.
%$$
%So,
 with the agreement that
$\frac{2\mu_{k}}{e^{2t\mu_{k}}-1}
\coloneqq t^{-1}$ if {$k=0$},
conditions (A1) and (A2) of Remark \ref{rem:exist} hold true if and only if
\begin{equation}\label{cond8}
\begin{split}
 \exists \,\eta\in L^1_{loc}([0,+\infty);\R) \  \mbox{bounded in a neighborhood of } +\infty \ \mbox{s.t.}\\
 \sqrt{
\frac{2\mu_{k} \left(\delta+\mu_k\right)^{\frac{1}{2}+2\varepsilon}}{(e^{2t\mu_{k}}-1)
q_k}
}\leq \eta(t), \ \ \ \forall t> 0, \ \forall k\in \N.
\end{split}
\end{equation}
%\begin{equation}\label{cond8}
%\sqrt{
%\frac{2\mu_{k} \left(\delta+\mu_k\right)^{\frac{1}{2}+2\varepsilon}}{(e^{2t\mu_{k}}-1)q_k}
%}\leq \gamma(t), \ \ \ \forall t> 0, \ \forall k\in \N.
%\end{equation}
Assume that
\begin{equation}\label{liminf}
\liminf_{k\to\infty}\frac{q_k}{k^{-2\theta}}>0 \ \ \mbox{for some} \ \theta\geq 0,
\end{equation} and let $k_0\in\N$ and $c_0>0$ be such that $q_k\geq c_0 k^{-2\theta} $ for some $c_0>0$ and every $k\geq k_0$. Considering \eqref{kj}, let $c_1,c_2>0$ and $k_0'\in \N$ be such that
$
c_1k^{\frac{2}{d}} \leq \mu_k \leq c_2k^{\frac{2}{d}}$ for every $k\geq k_0'.$
%Clearly (SF: E' COSI'? \red{FG: OF COURSE. PER QUESTI VALE LA STIMA CON $\eta=t^{-1/2}$})
{Calling $\bar{k}\coloneqq k_0\vee k_0'$ it is clear that,
for a suitable $C_0>0$,
$$
\sup_{k <\bar k}\sqrt{
\frac{2\mu_{k} \left(\delta+\mu_k\right)^{\frac{1}{2}+2\varepsilon}}{(e^{2t\mu_{k}}-1)
q_k}}\le C_0 t^{-1/2}.
$$
Hence, to prove \eqref{cond8} above, we  take $k \ge \bar k$ and we rewrite \eqref{cond8} (up to a constant depending on $c_0,c_1,c_2$) as}
% \begin{equation}\label{cond9}
%\sqrt{
%\frac{2\mu_{k} \left(\delta+\right)^{\frac{1}{2}+2\varepsilon+\theta}}{c_0(e^{2t\mu_{k}}-1)}
%}\leq \gamma(t), \ \ \ \forall t> 0, \ \forall k\geq k_1.
%\end{equation}
%Considering \eqref{kj}
% the above \eqref{cond9} rewrites as
\begin{equation}\label{cond10}
\begin{split}
 \exists \,\eta\in L^1_{loc}([0,+\infty);\R) \  \mbox{bounded in a neighborhood of } +\infty \ \mbox{s.t.}\\
\sqrt{
\frac{{k}^{\frac{2}{d}} \left(\delta+k^{\frac{2}{d}}\right)^{\frac{1}{2}+2\varepsilon}}
{(e^{2tk^{\frac{2}{d}}}-1)k^{-2\theta}}
}\leq \eta(t), \ \ \ \forall t> 0, \ \forall k\geq \bar{k}.
\end{split}
\end{equation}
%Assume that
%\begin{equation}
%\frac{3}{2}+2\varepsilon+\frac{d\theta}{2}\leq 1.
%\label{theta}
%\end{equation}
Noting that ${C_1}\coloneqq \sup_{s>0} \frac{s^{\frac{3}{2}+2\varepsilon+d\theta}}{e^s-1}<+\infty$, we can estimate
$$
\frac{{k}^{\frac{2}{d}} \left(\delta+k^{\frac{2}{d}}\right)^{\frac{1}{2}+2\varepsilon}}
{(e^{2tk^{\frac{2}{d}}}-1)k^{-2\theta}}\leq
\frac{(1+\delta)^{\frac{1}{2}+2\varepsilon}
k^{\frac{2}{d}\left(\frac{3}{2}+2\varepsilon\right)+2\theta}}
{(e^{2tk^{\frac{2}{d}}}-1)}\leq
{C_1} \frac{(1+\delta)^{\frac{1}{2}+2\varepsilon}}{(2t)^{\frac{3}{2}
+2\varepsilon+d\theta}}
\ \ \ \forall k\geq \bar{k}.
$$
Therefore,  (H2) is satisfied whenever \eqref{liminf} holds for some $\theta$ such that $\frac{3}{2}+2\varepsilon+d\theta<2$. As $\varepsilon>0$ can be taken arbitrarily small, we conclude that (H2) can be fulfilled if \eqref{liminf} holds for some $\theta$ such that
\begin{equation}\label{theta2}
\frac{3}{2}+d\theta<2 \ \Longleftrightarrow \ \theta<\frac{1}{2d}.
\end{equation}
\item \emph{On the simultaneous validity of (H1)--(H2).} Looking at \eqref{limsupbis} and \eqref{theta2}, we see that (H1)-(H2) can be simultaneously fulfilled by choosing a suitable $\varepsilon>0$ if $\sigma$ is diagonal with respect to $\{e_k\}_{k\in\N}$ and  \eqref{liminf} is verified for some $\theta\geq 0$ such that
\begin{equation}
\label{theta3}
\frac{1}{2}-\frac{1}{d}<\theta <\frac{1}{2d}.
\end{equation}
These requirements can be fulfilled only for dimension $d\leq 2$.
\item \emph{On the validity of (H3).} This is guaranteed, for instance, if
 $\Lambda$ is bounded,
 $\ell_1$ is continuous and bounded,
$\ell_2$ is measurable.
%\item[-] $|\ell(x_1,u)-\ell(x_2,u)|\leq C_0|x_1-x_2|_H$, for every $x_1,x_2\in H,$ for some $C_0\geq 0$.
\end{itemize}
%Also Corollary \ref{cr4:optimalfeedbackell} applies if
%we know that, in this case the feedback map $\Phi$
%defined in (\ref{eq4:feedbackregularell}) admits a measurable selection
%$\phi:H \to U$ and consider the SDE (CLE)
%\begin{equation}
%\label{eq4:CLEselectionell}
%\begin{itemize}
%\item[(A3)] $l_2$ is strictly convex.
%\end{itemize}
%Then
%is actually single-valued.
% Assume, moreover, that

\subsubsection{Optimal Feedback Controls} \label{FEEDBACKNEUMANN}
In the framework of the previous subsection, we look now at the existence of optimal feedback controls.
\begin{Theorem}
Let (H1)--(H3) of the previous subsection hold. Assume that the multi-valued map
\begin{equation}\label{pphi}
\Psi: H\to \Lambda, \ \ q\mapsto 	\textsl{arg}\!\!\!\!\!\!\min_{u \in \Lambda\ \ \ \ \ } \left\{\left\langle  L_N^{\delta, \varepsilon} u, q \right\rangle_H +l_2(u)  \right\}
\end{equation}
admits a Lipschitz continuous selection $\psi$ and that  $D^{G_N^{\delta, \varepsilon}}v$ is Lipschitz continuous. Set $\phi:=\psi\circ D^{G_N^{\delta, \varepsilon}}v$.
Then the SDE
\begin{equation}
\label{eq4:CLEselectionelleq}
\left \{
\begin{array}{l}
{d} X(t) = \left[A_NX(t)+ G_N^{\delta, \varepsilon}L_N^{\delta, \varepsilon}(\phi (X(t))) \right]\,{d} t +
\sigma\, {d} W(t), \ \ \ \ t\geq 0,\\[5pt]
X(0)=x,
\end{array}
\right.
\end{equation}
admits a unique mild solution $X_\phi(\cdot;x)\in \mathcal{K}_\calp^{1,loc}(H)$ (in the sense of \eqref{eq4:CLEselectionellbis}) admitting a version with continuous trajectories. As a consequence, Corollary \ref{cr4:optimalfeedbackell}(i) applies providing the optimality {of the couple $\big({u_\phi}(\cdot),{X_\phi}(\cdot;x)\big)$, where $u_\phi (t):=\phi(X_\phi(t;x))$ for $t\geq 0$.}
\end{Theorem}
\begin{proof}
By the assumptions, the  map $\phi$
is Lipschitz continuous too. Then the proof follows the classical fixed point arguments as in standard results of existence and uniqueness of SDEs in infinite dimension, see e.g. \cite[Theorem 7.5]{DaPratoZabczyk14}. Here we only need to take care of dealing with $\overline{e^{sA_N}G_N^{\delta,\varepsilon}}$ in place of $e^{sA_N}$ {in the convolution term} and use \eqref{esg} with $G=G_N^{\delta,\varepsilon}$.
\end{proof}

The assumption that $\Psi$ defined in \eqref{pphi} admits a Lipschitz continuous selection $\psi$ is guaranteed, for example, if
{$\Lambda=U$, $l_2: U\to \R$ is strictly convex,
$$
\lim_{|u|_U\to+\infty} \frac{l_2(u)}{|u|_U}=+\infty,
$$  $l_2$ is Fr\'echet differentiable, and $D l_2$ has Lipschitz continuous inverse.}
Indeed, in this case the infimum in \eqref{pphi} is uniquely achieved (hence, $\Psi$ is single-valued)  at
$$u^*(q)= (Dl_2)^{-1}\left(\big(L_N^{\delta, \varepsilon}\big)^*q\right), \ \ \ q\in H.$$
%(see, e.g., CITARE CANNARSA SINESTRARI Theorem A.1.3)}.
 Hence, if we are able to check that $D^{G_N^{\delta, \varepsilon}}v$ is Lipschitz continuous,  we can then apply Corollary \ref{cr4:optimalfeedbackell}(i) in its strongest form to get uniqueness of the optimal control constructed.

 On the other hand, checking that $D^{G_N^{\delta, \varepsilon}}v$ is Lipschitz continuous might be, in general, a very difficult task\footnote{This can be done assuming more regularity of $\ell_1$ --- hence of $l_1$ --- and proving a suitable $C^2$ property of $v$. See, e.g., the approach used in \cite{GozziRouy96} or in \cite{GozziMasiero12}.}, whereas mere continuity of $D^{G_N^{\delta, \varepsilon}}v$ is a condition already ``contained'' in the definition of mild solution  to \eqref{eq2:HJDirichletNeumann}.
Hence,  it would be meaningful to provide a Peano type result
 \footnote{This is not straightforward: in infinite dimension Peano's Theorem fails in general (see \cite{Godunov}).}
 of existence of mild solutions to CLE \eqref{eq4:CLEselectionelleq}. This seems possible when a selection $\psi$ of $\Psi$ in \eqref{pphi} is  known to be only continuous and bounded on bounded sets, as
 \begin{itemize}
 \item[(i)]  the semigroup $\{{e^{tA_N}}\}_{t\geq 0}$  is compact;
 \item[(ii)] as  $D^{G_N^{\delta, \varepsilon}}v$ is continuous and bounded by construction,  the map $\phi:=\psi\circ D^{G_N^{\delta, \varepsilon}}v$ is continuous and bounded.
\end{itemize}
 Indeed, in such a framework, it seems possible to use the methods of \cite[Prop.\,3]{ChojnowskaGoldys95} (see also \cite{GatarekGoldys97}), passing through the use of the so-called Skorohod representation theorem, to construct
martingale solutions to \eqref{eq4:CLEselectionelleq}; hence, to construct optimal feedback controls in the weak formulation.
% Here, we only need to take care of dealing with $\overline{e^{sA_N}G_N^{\delta,\varepsilon}}$ in place of $e^{sA_N}$ {in the convolution term}.
%\begin{Theorem}\label{th:pean0}
%Let (H1)--(H3) of the previous subsection hold. Assume that the multi-valued map
%\eqref{pphi}
%admits a  continuous selection bounded on bounded sets and  set $\phi:=\psi\circ D^{G_N^{\delta, \varepsilon}}v$.
%Then
%\eqref{eq4:CLEselectionelleq}
%has a martingale solution.
%% $\overline{X_\phi}(\cdot;x)$ in some filtered probability space
%%$\left(\overline{\Omega},\overline{\mathcal{F}},\{\overline{\mathcal{F}_t}\}_{t\geq 0}, \overline{\P} \right)$
%%and for some cylindrical Brownian motion $\overline{W}$ defined on it.
%As a consequence, Corollary \ref{cr4:optimalfeedbackell}(ii) applies providing the optimality {of the couple $(\overline{u_\phi}(\cdot),\overline{X_\phi}(\cdot;x))$, where $\overline{u}_\phi (s):={\phi}(\overline{X_\phi}(s;x))$.}
%\end{Theorem}

\begin{Remark}\label{rem:peano}
%The method to prove the existence of martingale solutions to \eqref{eq4:CLEselectionelleq} relies on the use of the so-called Skorohod representation theorem. This explains the need of dealing with martingale solutions, i.e. of constructing the solution on a probability space $\overline{\Omega}$ which is not the original one of the definition of the problem.
In the specific case we are handling, where the diffusion term is just additive in the equation, {a way to
%go around this and to
construct the solution in the original probability space $\Omega$} might consist in constructing a pathwise solution dealing with a parameterized family of deterministic problems with parameter $\omega\in \Omega$ (see \cite{BensoussanTemam73},  \cite[Sections 14.2 and 15.2]{DaPratoZabczyk96}, \cite{FlandoliSchmalfuss99}, \cite{MikuleviciusRozovski05}). {Once this is done, the} problem is to prove that the family of solutions constructed $\omega$ by $\omega$ admits an adapted selection. The existence of a selection measurable with respect to $\mathcal{F}$ can be obtained using measurable selection {theorems} (see again \cite{BensoussanTemam73}); proving that this selection is also adapted is a problematic task, {which is still open. In the case when one knows ex ante that the pathwise solution is unique for a.e. $\omega \in \Omega$, then F. Flandoli (personal communication) showed us how to accomplish this task. Unfortunately, in our case, the uniqueness of the solutions of the deterministic equations for a.e. $\omega \in \Omega$ only holds when the properties of the coefficients allow to find directly mild solutions to SDE \eqref{eq4:CLEselectionelleq}.}
 \end{Remark}

\subsection{Stochastic optimal control with delay in the control variable}\label{sec:delay}

Here we consider an infinite horizon version of a control problem studied in \cite{GozziMasiero12,GozziMasiero16b}.
%Let $\left(\Omega,\mathcal{F},\{\mathcal{ F}_t\}_{t\geq 0},\P \right)$
%be a filtered probability space and c
Consider the following linear controlled one dimensional SDE:
\begin{equation}
\begin{cases}
dy(t)  =\Big[a_0 y(t) +b_0 u(t)  +\dis\int_{-d}^0b_1(\xi)u(t+\xi)d\xi\Big]\,dt+\sigma_0 \,dW(t)
,\text{ \ \ \ }t\ge 0, \smallskip\\
y(0)  =y_0,\ \ \
u(\xi)=u_0(\xi), \quad \xi \in [-d,0),
\end{cases}  \label{eq-contr-ritINTRO}
\end{equation}
where \begin{itemize}
\item $W=\{W(t)\}_{t\geq 0}$ is a standard one dimensional Brownian motion;
\item $a_0,b_0,\sigma_0\in \R$, $\sigma_0 >0$; \item $d>0$ represents the maximum delay
the control takes to affect the system;
\item $b_1(\cdot)$ is a (real-valued)  function
weighting the aftereffects of the control on the system; we consider here the case of distributed delay, i.e. when $b_1\in L^2([-d,0],\R)$. 
\end{itemize}

%while a more difficult case which we leave for further research is when $b_1$ is a measure, e.g. a Dirac delta in $-d$ (``pointwise delay'').

The initial data are the initial state $y_0$ and the past history $u_0$ of the control.
The control $u$ takes values in a closed subset $\Lambda\subseteq  U:=\R$ and belongs to $\mathcal{U}_2$ (defined by \eqref{up} with $p=2$).

Such kind of equations (even in a deterministic framework) have been used  to model the effect of advertising on
the sales of a product  \cite{GozziMarinelli06,GozziMarinelliSavin09,  FedericoTacconi14}}, the effect of investments with time to build on growth \cite{FabbriGozzi08, Federico15}, to model optimal portfolio problems with execution delay
\cite{BruderPham09},  to model the interaction of drugs with tumor cells
\cite[p.\,17]{KolmanovskiShaiket96}.

%In many applied cases (like the ones quoted above)
Denoting by $y^{y_0,u_0, u(\cdot)}$ the unique solution to \eqref{eq-contr-ritINTRO}, the goal of the problem is to minimize, over all control strategies in $\mathcal{U}_2$, the following objective functional
\begin{equation}\label{costoconcretoINTRO}
\E \left[\int_0^{\infty} e^{-\lambda t} \left(\ell_0(y^{y_0,u_0,u(\cdot)}(t))+\ell_1(u(t))\right)\,dt\right],
\end{equation}
where $\ell_0:\R\rightarrow \R$ and $\ell_1:\Lambda\rightarrow \R$ are measurable and bounded from below.
It is important to note that here $\ell_0$ and $\ell_1$ do not depend on the past of the state and/or control. This is a very common feature of {many} applied problems.

A standard way
to approach these delayed control problems, introduced in \cite{VinterKwong81} for the deterministic case and extended to the stochastic case  in  \cite{GozziMarinelli06}, is to reformulate them as  equivalent infinite dimensional control problems without delay\footnote{It must be noted that, under suitable restrictions on the data, one can treat (stochastic) optimal control
problems with delay avoiding to look at them as infinite dimensional systems (see \cite{FedericoTankov15}).
However, this is possible only in  very special cases,
 leaving out a lot of of concrete applications.}.
The details are given in \cite{GozziMasiero12} for the finite horizon case, which is completely similar to the infinite horizon case, with the obvious changes (see also \cite{FedericoTacconi14} for the infinite horizon case in a deterministic framework with a different embedding space).
Consider  the Hilbert space $H:=\R \times L^2([-d,0],\R)$, set $b:=(b_0,b_1(\cdot))\in H$, and assume, without loss of generality,  $|b|_H=1$. The state equation \myref{eq-contr-ritINTRO} is rephrased in $H$ as a linear SDE  with state variable $X=(X_0,X_1(\cdot))$ as follows:
\begin{equation}
\begin{cases}
dX(t)  =\big[AX(t) +Gu(t) \big]\,dt+\sigma dW(t)
,\text{ \ \ \ }t\ge 0, \\
X(0)={x}=(x_0,x_1(\cdot)),
\end{cases}   \label{eq-astrINTRO}
\end{equation}
where
$$ \mathcal{D}(A)=\big\{(x_0,x_1(\cdot))\in \R\times W^{1,2}([-d,0]): \ x_1(-d)=0\big\}, \ \ \ Ax=\left(a_0x_0+x_1(0),\;-x_1'\right);$$
$$
G:\R\rightarrow H,\quad G(u)= u{b};
\qquad
\sigma:\R\rightarrow H,\quad \sigma(z)=(\sigma_0 z, 0);
$$
 and
the initial datum ${x}$ is defined as
$$x_0:=y_0, \ \ \ x_1(\xi):=\int_{-d}^\xi b_1(\varsigma)u_0(\varsigma-\xi)d\varsigma, \  \ \xi \in [-d,0].$$
It is well known that  $A$ is the generator of a $C_0$-semigroup of linear bounded operators on $H$.
Note that the infinite dimensional datum $x_1(\cdot)$ depends on the ``initial past'' $u_0(\cdot)$ of the control. It turns out that $X_0(t;x,u(\cdot))=y^{y_0,u_0,u(\cdot)}$, so
 \eqref{costoconcretoINTRO} is rewritten as
\begin{equation}\label{ppl1}
J(x;u(\cdot)):= \E\left[ \int_0^{\infty} e^{-\lambda t} \big(\ell_0(X_0(t;x,u(\cdot))+\ell_1(u(t))\big)\,dt\right].
\end{equation}
%Hence, the value function of the equivalent infinite dimensional problem is defined as
%$$
%V(x):=\inf_{u(\cdot)\in \mathcal{U}_2}
%\E \left[\int_0^{+\infty} \left[\ell(X_0(s))+\ell_1(u(s))\right] ds
%\right]
%$$
Setting $Q:=\sigma\sigma^*$, the HJB equation associated to the minimization of \eqref{ppl1} is
\begin{equation}\label{HJBINTRO}
\lambda v(x)
=
\frac{1}{2}\mbox{Tr}\left[QD^2v(x)
\right]+ \< Ax,D v(x)\>_H
+ \inf_{u\in \Lambda} \Bigg\{u
%\frac{\partial v}{\partial b}
D^Gv(x)+\ell_1(u)\Bigg\} +\ell_0(x_0), \ \ \ x\in H,
\end{equation}
%\begin{equation}\label{HJBINTRO}
%\lambda v(x)
%=
%\frac{1}{2}\mbox{Tr}\left[QD^2v(x)
%\right]+ \< Ax,D v(x)\>_H
%+ F_{0} (v_{x_0}(x)) +\ell_0(x_0), \ \ \ x\in H,
%\end{equation}
%where
%\begin{equation}\label{HminINTRO}
%F_{0} (q):=
%\inf_{u\in U} F_{0,CV}(q ;u),
%%=\inf_{u\in \Lambda} \{\<p,Bu\>_\calh+\ell_1(u) \}.
%\ \ \ \
%F_{0,CV}(q;u):=qu+\ell_1(u), \ \ \ q\in \R, \ u\in\Lambda.
%%=\<B^*p,u\>_\R +\ell_1(u).
%\end{equation}
Notice that  $D^G=\frac{\partial}{\partial {b}}$, where the latter symbol denotes the directional derivative along the direction $b$. So,  the nice feature of the equation above is that the nonlinearity on the gradient only involves the directional derivative $D^G$.
%where by $\mathbf{0}$ we have denoted the null element of $L^2([-d,0];\R)$.
Note also that here we do not have the so called \emph{structural condition}  $G(\R)\subseteq \sigma (\R)$; this prevents the use of techniques based on Backward SDEs (see, e.g., \cite{FuTe-ell}) to tackle the problem.
%\footnote{More general results could be proved for such equation: for example, using the technique of \cite{Cerrai01-40} one can deal with the case when $F_0$ is only locally Lipschitz continuous. Also the case of unbounded $\bar \ell_0$ can be treated, see on this Remark \ref{rm:polgrowth}.}.
%The mild solution (defined through the usual integral form, see Definition \ref{df4:solmildHJBell}) belongs to $D(\cala)\cap C_b^{1,G}(H)$.

Now we check if the assumptions of our main result Theorem \ref{th:ver} are verified. First of all, it is easy to check that
Assumption \ref{hp4:ABQforOU} and Assumption \ref{ass:Gbis} hold.
%clearly verified in this case. Hence, as
%also Assumption \ref{ass:Gbis} holds, 
The third assumption, i.e.  the existence of a mild solution $v\in\mathcal{S}^{A,G}(H)$ to \myref{HJBINTRO} needs to be discussed. 

In \cite{GozziMasiero12}, the authors study a finite horizon optimal control problem with the same state equation \eqref{eq-contr-ritINTRO} and a similar objective functional. Exploiting only partial smoothing properties of {the transition semigroup associated to the state equation \eqref{eq-astrINTRO} with null control,}
%generated by the operator $A$, 
the authors are able to provide, under suitable reasonable assumptions on the data, existence and uniqueness results for the parabolic HJB equation associated to the control problem.

{We believe that the approach of \cite{GozziMasiero12} can be adapted to our infinite horizon case, getting a mild solution $v\in \mathcal{D}(\mathcal{A}^{(0)})\cap C^{1,G}_b(H)$
to HJB  \eqref{HJBINTRO}.
Then, to apply our theory one should prove that such function $v$ is Lipschitz continuous on compact sets, which enables to apply Proposition \ref{lemma2} to get $v\in\mathcal{S}^{A,G}(H)$. To  get this goal one can  proceed as in
\cite{GozziMasiero12} by assuming more regularity on the data of the problem.
More precisely, assuming that $l_0\in C^1_b(\R)$ and that
the Hamiltonian $p\mapsto \inf_{u\in \Lambda} \left\{u p+\ell_1(u)\right\}$ is differentiable with Lipschitz continuous derivative,
\cite{GozziMasiero12} proves that the mild solution $v\in C^1_b(H)$.
This fact, in particular, implies the required Lipschitz continuity of $v$.
In \cite{GozziMasiero16b} the authors also provide a verification theorem for their finite horizon problem. They  use an approximation procedure of the solution of the HJB equation, which our results allow to avoid here.
}

\appendix
\section{Appendix}\label{sec:appendix}
Recall that, given $G\in \call_u(K,H)$, the pseudo-inverse $G^{-1}:\calr(G)\to\cald(G)$ is defined as the operator that associates to each $h\in \calr(G)$ the element of $G^{-1}(\{h\})$ having minimum norm.\footnote{Existence and uniqueness of such an element follows from the fact that $G$ is a closed operator and applying the results of \cite[Sec.\,II.4.29, p.\,74]{DunfordSchwartz58-I}). } Note that
%, as $\calr(G)$ is closed and $G$ is closed, by the closed graph theorem $G^{-1}$ is bounded. Also,
$G^{-1}G:\cald (G)\to \cald(G)$ is bounded, so it can be extended to a bounded operator  $\overline{G^{-1}G}\in \call(K)$.
{\begin{Lemma}\label{lemma:new} We have
\begin{equation}\label{kkk}
\left\langle D^Gf(x), \overline{G^{-1}G}k\right\rangle_K=\left\langle D^Gf(x), k\right\rangle_K, \ \ \ \forall k\in K, \; {\forall x \in H}.
\end{equation}
\end{Lemma}
\begin{proof}
Assume first that $k\in\cald(G)$. In this case $GG^{-1}Gk=Gk$. Then, using Remark \ref{rem:mmn}, we write
\begin{align*}
\left\langle D^Gf(x), \overline{G^{-1}G}k\right\rangle_K&=
\lim_{s\to 0} \frac{f\big(x+sG\overline{G^{-1}G}k\big)-f(x)}{s}\\
&=
\lim_{s\to 0} \frac{f\big(x+sGk\big)-f(x)}{s}\\
&=\left\langle D^Gf(x), k\right\rangle_K, \ \ \  \; {\forall x \in H}.
\end{align*}
If $k\notin\cald(G)$, we can take a sequence $\{k_n\}\subseteq \cald(G)$ converging to $k$. Considering \eqref{kkk} on $k_n$ and passing to the limit the claim follows taking into account that $\overline{G^{-1}G}$ is bounded.
\end{proof}}

\begin{Assumption}\label{ass:G}
The operator $G\in\call_u(K,H)$
is such that  {for every $k\in K$
\begin{itemize}
\item[(i)] there exists $\varepsilon>0$ such that $\left\{\dis\int_0^t\overline{e^{sA}G}k\, ds
%\in \calr(G)
\right\}_{t\in(0,\varepsilon)}\subseteq \calr(G)$;
% the following.
%\begin{enumerate}[(i)]
%\item $\left\{\int_0^t\overline{e^{sA}G}k ds
%%\in \calr(G)
%\right\}_{t\in(0,\varepsilon)}\subset \calr(G)$ for sufficiently small $\varepsilon>0$.
%\item $\left\{ G^{-1} \frac{1}{t}\int_0^t\overline{e^{sA}G}k ds\right\}_{t\in(0,\varepsilon)}$ is bounded for sufficiently small $\varepsilon>0$; \red{RIDONDANTE: CONTENUTA NELLA SUCCESSIVA}
\item[(ii)]
$G^{-1} \left(\frac{1}{t}\dis\int_0^t\overline{e^{sA}G}k\, ds\right)\to \overline{G^{-1}G}k$, as $t\to 0^+$.
\end{itemize}
}
\end{Assumption}
\begin{Remark}\label{rem:imppo}
Note that $\dis\int_0^te^{sA}h\, ds\in \cald(A)$ for every  $t>0$ and $h\in H$.
% Indeed, by Remark \ref{Rem:ext}, we have
% $$
%\int_0^r\overline{e^{sA}G}k ds= \int
% $$
%
%
So, in view of the fact that $\overline{G^{-1}G}$ is bounded, %closed \blu{NON VEDO DOVE SI USA CHE $G$ E' CHIUSO: SE MAI SI USA CHE $G^{-1}G$ E' LIMITATO},
Assumption \ref{ass:G} is verified, in particular, if $K=H$,  $\cald(A)\subseteq \cald(G)$ and, for sufficiently small $\varepsilon>0$,
$$
G\int_0^te^{sA}h ds = \int_0^t\overline{e^{sA}G}h ds, \ \ \ \forall t\in(0,\varepsilon), \ \forall h\in H.
$$ This applies, e.g., to the case when $A$ is dissipative and generates an analytic semigroup,  and $G=(\delta I-A)^\beta$ with $\delta> 0$ and $\beta\in (0,1)$ (see the example of Section \ref{SS:NEUMANN}).
\end{Remark}

\begin{Proposition}\label{lemma1}
Let Assumption \ref{ass:G} holds. Then $\mathcal{S}^{A,G}(H)=C^{1,G}_b(H)$.
\end{Proposition}

\begin{proof}
{Fix $k\in K$, $z\in C(\R^+;H)$ and let $\varepsilon>0$ be as in Assumption \ref{ass:G}(i). }
%Let $G^{-1}:\calr (G)\subseteq H\to \cald(G)$ be the pseudo-inverse of $G$.
 Noting that $GG^{-1}h=h$ for every $h\in \calr(G)$,
 by Assumption \ref{ass:G}(i) we can write
\begin{equation}\label{ass1}
\int_0^t\overline{e^{sA}G}k\, ds= G k(t), \ \ \mbox{where} \ \
k(t):=G^{-1}\int_0^t\overline{e^{sA}G}k\, ds, \ \ \ \forall t\in(0,\varepsilon).
\end{equation}
Moreover, by Assumption \ref{ass:G}(ii), we have
%and
\begin{equation}\label{ass3}
\frac{k(t)}{t}\ \stackrel{t\to 0^+}{\longrightarrow}\  \overline{G^{-1}G}k.
\end{equation}
%In particular
%\begin{equation}\label{ass2}
%\sup_{t\in(0,\varepsilon)}\frac{|k(t)|_K}{t}<+\infty.
%\end{equation}
{Fix now $t\in(0,\varepsilon)$.} Using \eqref{ass1} we write
\begin{align}\label{aa4bis}
\frac{\varphi\left( z(t)+\dis\int_0^t\overline{e^{sA}G}k \,ds\right)- \varphi(z(t))}{t}
=&\,\,
\frac{\varphi\big( z(t)+Gk(t)\big)-\varphi(z(t))-\<D^Gf(x(t)),k(t)\>_K}{t}
%|k(t)|_K}\cdot %\frac{|k(t)|_K}{t}
\nonumber\\&
+
\<D^G\varphi(z(t)),\frac{k(t)}{t}\>_K.
\end{align}
Mean value theorem applied to the function 
$[0,1]\to \R, \ \xi\mapsto f\left(x(t)+\xi Gk{(t)}\right)$ yields {(see also Remark \ref{rem:mmn})}
\begin{align*}
\varphi\big(z(t)+Gk(t)\big)-\varphi(z(t))&={\int_0^1 \frac{d}{d\xi} \varphi\big(z(t)+\xi G k(t)\big)d\xi}\\
&{=\int_0^1 \lim_{{\eta}\to 0} \frac{\varphi\big(z(t)+(\xi+{\eta})Gk(t)\big)-\varphi\big(z(t)+\xi G k(t)\big)}{{\eta}}d\xi}\\
&=\int_0^1\left\langle D^G\varphi\big(z(t)+\xi Gk(t)\big),k(t)\right\rangle_K d\xi.
\end{align*}
 Hence, \eqref{aa4bis} rewrites as
 \begin{align}\label{aa4}
\frac{\varphi\left( z(t)+\dis\int_0^t\overline{e^{sA}G}k ds\right)- \varphi(z(t))}{t}
=&
\int_0^1\< D^G\varphi\big(z(t)+\xi Gk(t)\big)-D^G\varphi(z(t)),\frac{k(t)}{t}\>_Kd\xi 
%{|k(t)|_K}\cdot \frac{|k(t)|_K}{t}
\nonumber\\&
+
\<D^G\varphi(z(t)),\frac{k(t)}{t}\>_K.
\end{align}
%there exists $\xi\in[0,1]$ such that
%$$f(x+Gk)-f(x)=\left\langle D^Gf(x+\xi Gk),k\right\rangle_K.$$
Moreover, we can estimate
\begin{equation}\label{aa9}
\left|\< D^G\varphi\big(z(t)+\xi Gk(t)\big)-D^G\varphi(z(t)),\frac{k(t)}{t}\>_K \right|
%{|k(t)|_K}
\leq \left| D^G\varphi\big(z(t)+\xi Gk(t)\big)-D^G\varphi(z(t))\right|_K\cdot
\left|\frac{k(t)}{t}\right|_K
 \ \forall \xi\in[0,1].
\end{equation}
{Now we are going to take the limit for $t\to 0^+$ in \eqref{aa4}.} {To this purpose, we observe that,} as $D^G\varphi\in C_b(H,K)$  and
 $\big\{z(t)\big\}_{t\in(0,\varepsilon)}$ is compact in $H$, we have
 \begin{equation}\label{aa8}
\sup_{t\in(0,\varepsilon)}\left| D^G\varphi\big(z(t)+h\big)-D^Gf(z(t))\right|_K\\ \stackrel{|h|\to0^+}{\longrightarrow} 0.
\end{equation}
{By definition of $k(t)$ (see \eqref{ass1}), we have $|Gk(t)|_H\stackrel{t\to 0^+}{\longrightarrow} 0$. Hence, \eqref{aa8} provides
\begin{equation}\label{aa10}
\sup_{\xi\in[0,1]}\left| D^G\varphi\big(z(t)+\xi Gk(t)\big)-D^G\varphi(z(t))\right|_K
\stackrel{t\to 0^+}{\longrightarrow} 0.
\end{equation}
%\begin{equation}\label{aa5}
%\frac{\int_0^1\< D^Gf(x(t)+\xi Gk(t))-D^Gf(x(t)),k(t)\>_Kd\xi }{|k(t)|_K}
%\  \stackrel{t\to0^+}{\longrightarrow} \ 0.
%\end{equation}}
%$$
%\limsup_{t\to 0^+} \sup_{r\in(0,\varepsilon)}\frac{\left|f\left( x(r)+Gk(t)\right)- f(x(r))-\< D^Gf(x(r)), k(t)\>\right|}{|k(t)|}=0.
%$$
Hence, combining \eqref{ass3}, \eqref{aa9} and \eqref{aa10}, we get
\begin{equation}\label{aa5}
\int_0^1\< D^G\varphi\big(z(t)+\xi Gk(t)\big)-D^G\varphi(z(t)),\frac{k(t)}{t}\>_Kd\xi 
\stackrel{t\to 0^+}{\longrightarrow} 0.
\end{equation}}
% the first term of the right hand side of \eqref{aa4}
%converges to $0$ as $t \to 0^+$.}
%\begin{equation}\label{aa5}
%\frac{f\left( x(t)+Gk(t)\right)- f(x(t))-\<D^Gf(x(t),k(t)\>_K}{|k(t)|_K}\cdot \frac{|k(t)|}{t}\ \stackrel{t\to0^+}{\longrightarrow}\  0.
%\end{equation}
Moreover, \eqref{ass3} and
%Hence,  by Assumption \ref{ass:G}(iii) and
the continuity of the maps $t\mapsto z(t)$ and  $x\mapsto D^G\varphi(x)$ entails
\begin{equation}\label{aa6}
 \<D^G\varphi(z(t)),\frac{k(t)}{t}\>_K \ \stackrel{t\to0^+}{\longrightarrow} \ \left\langle D^G\varphi(z(0)), \overline{G^{-1}G}k\right\rangle_{{K}}.
\end{equation}
Combining \eqref{aa4}, \eqref{aa5},  \eqref{aa6}, and Lemma \ref{lemma:new}, the claim follows.
\end{proof}

\begin{Assumption}\label{ass:Gbis}
 $G\in\call(K,H)$.
\end{Assumption}
\begin{Proposition}\label{lemma2} Let Assumption \ref{ass:Gbis} hold and let $\varphi\in C^{1,G}_b(H)$ be  Lipschitz continuous on compact sets. Then $\varphi\in \mathcal{S}^{A,G}(H)$.
\end{Proposition}
\begin{proof}
Let $k\in K$. 
Observe that, as $G\in\call(K,H)$, we have $k\in K=\cald(G)$, $\overline{e^{sA}G}k= e^{sA}Gk$ for every $s>0$, and
\begin{equation}\label{bdd}
\lim_{t\to 0^+}\frac{1}{t}\int_0^t{e^{sA}G}k ds\to Gk.
\end{equation}
Let $t>0$. We can split
%\begin{eqnarray}\label{aa4tris}
%&\left|\frac{\varphi\left( z(t)+\int_0^t e^{sA}Gk ds\right)- \varphi(z(t))}{t}- \<D^G\varphi(z(0)),k\>_K\right|\\
%&\leq
%\left|\frac{\varphi\left( z(t)+\int_0^t{e^{sA}G}k ds\right)- f(x(t)+tGk)}{t}\right|\nonumber
% + \left|\frac{f(x(t)+tGk)-f(x(t))}{t}-\<D^Gf(x(0)),k\>_K\right|.
%\end{eqnarray}
\begin{align}\label{aa4tris}
&\frac{\varphi\left( z(t)+\dis\int_0^t e^{sA}Gk ds\right)- \varphi(z(t))}{t}\nonumber
\\&=
\frac{\varphi\left( z(t)+\dis\int_0^t{e^{sA}G}k \,ds\right)- \varphi\big(z(t)+tGk\big)}{t}
 + \frac{\varphi\big(z(t)+tGk\big)-\varphi(z(t))}{t}.
\end{align}
The set $\bigg\{ z(t)+\dis\int_0^t{e^{sA}G}k\, ds\bigg\}_{t\in(0,1)}\bigcup\,\bigg\{z(t)+tGk\bigg\}_{t\in(0,1)}\subset K$ is precompact. Hence, by Lipschitz continuity of $\varphi$ on compact sets,
%Now note that, by Lipschitz continuity of  $f$, we have
%$$
%\sup_{k\in \mathcal{K},t\in[0,1]} \left[\left|x(t)+\int_0^t{e^{sA}G}k ds\right|+\left|x(t)+tGk\right|\right]
%<\infty,$$
we have for some $C_0>0$ independent of $t\in(0,1)$
\begin{equation}\label{aabis}
\begin{split}
\left|\frac{\varphi\left( z(t)+\dis\int_0^t{e^{sA}G}k ds\right)- \varphi\big(z(t)+tGk\big)}{t}\right|\leq
  C_0 \left|\frac{1}{t}\int_0^t{e^{sA}G}k ds-Gk\right|.
 \end{split}
\end{equation}
We let now  $t\to 0^+$ in  \eqref{aa4tris}. Combining with \eqref{aabis} and \eqref{bdd} we get
\begin{equation}\label{final}
\lim_{t\to 0^+}\frac{\varphi\left( z(t)+\dis\int_0^t e^{sA}Gk ds\right)- \varphi(z(t))}{t}
= \lim_{t\to 0^+}
 \frac{\varphi\big(z(t)+tGk\big)-\varphi(z(t))}{t},
\end{equation}
provided that the limit in the right hand side above exists, as we are going to show.
We write
\begin{align*}
\varphi\big(z(t)+tGk\big)-\varphi(z(t))&={\int_0^1 \frac{d}{d\xi} \varphi\big(z(t)+\xi tG k\big)d\xi}\\
&{=\int_0^1 \lim_{{\eta}\to 0} \frac{\varphi\big(z(t)+(\xi+{\eta})tGk\big)-\varphi\big(z(t)+\xi tGk\big)}{{\eta}}d\xi}\\
&=\int_0^1\left\langle D^G\varphi\big(z(t)+\xi tGk\big),tk\right\rangle_K d\xi.
\end{align*}
By the equalities above and considering that   $D^G\varphi\in C_b(H;K)$, we have
$$\lim_{t\to 0^+}
 \frac{\varphi\big(z(t)+tGk\big)-\varphi(z(t))}{t}= \lim_{t\to 0^+} \int_0^1\left\langle D^G\varphi\big(z(t)+\xi tGk\big),k\right\rangle_K d\xi= \left\langle D^G\varphi(z(0)),
 k\right\rangle_K$$
 and the claim follows from \eqref{final}.

\end{proof}

\end{document}